\documentclass[11pt]{article}
\usepackage[margin=1in]{geometry}

\usepackage{amsmath}
\usepackage{amsthm}
\usepackage{amssymb}
\usepackage{comment}
\usepackage{dsfont}
\usepackage{enumerate}
\usepackage{tikz}
\usepackage{hyperref}
\usepackage[capitalize]{cleveref}
\usepackage{setspace}
\usepackage{graphicx}
\usepackage{subcaption}
\usepackage[normalem]{ulem}

\usepackage{algpseudocode, algorithm}
\newcommand*\Let[2]{\State #1 $\gets$ #2}
\algrenewcommand\algorithmicrequire{\textbf{Given:}}
\algrenewcommand\algorithmicensure{\textbf{Return:}}

\usepackage{authblk}
\usepackage{pgfplots}

\renewcommand{\algorithmicrequire}{\textbf{Input:}}
\renewcommand{\algorithmicensure}{\textbf{Output:}}

\newcommand{\commentout}[1]{}

\newcommand{\ba}{\boldsymbol{a}}
\newcommand{\bA}{\boldsymbol{A}}
\newcommand{\bb}{\boldsymbol{b}}
\newcommand{\bB}{\boldsymbol{B}}

\newcommand{\be}{\boldsymbol{e}}

\newcommand{\bP}{\boldsymbol{P}}

\newcommand{\bQ}{\boldsymbol{Q}}

\newcommand{\bU}{\boldsymbol{U}}
\newcommand{\bv}{\boldsymbol{v}}
\newcommand{\bV}{\boldsymbol{V}}

\newcommand{\bz}{\boldsymbol{z}}

\newcommand{\bSigma}{\boldsymbol{\Sigma}}

\newcommand{\calO}{\mathcal{O}}

\renewcommand{\hat}{\widehat}
\renewcommand{\tilde}{\widetilde}
\renewcommand{\epsilon}{\varepsilon}

\def\<{\big\langle}
\def\>{\big\rangle}
\def\({\Big(}
\def\){\Big)}

\def\C{\mathbb{C}}

\def\N{\mathbb{N}}

\def\R{\mathbb{R}}

\def\calU{\mathcal{U}}
\def\calV{\mathcal{V}}

\newtheorem{theorem}{Theorem}[section]

\newtheorem{lemma}[theorem]{Lemma}

\theoremstyle{definition}

\numberwithin{equation}{section}

\title{Efficient computation of the singular value decomposition \\ with linear photonic circuits}
\author{Johannes Maly$^{\star,\dagger}$,Korbinian Neuner$^\star$, Samarth Vadia$^{\circ,\ddagger}$\footnote{Authors are named in alphabetical order, cf.\ author contribution statement below. Corresponding author: Johannes Maly.}}
\affil{$^\star$ Department of Mathematics, Ludwig-Maximilians-Universität München, Germany \\
$^\dagger$ Munich Center for Machine Learning (MCML), Germany \\
$^\circ$ Linque GmbH, Martiusstr. 5, 80802 Munich, Germany \\
$^\ddagger$ Institute of Informatics, Ludwig-Maximilians-Universität München, Germany}

\date{}

\begin{document}

\maketitle

\begin{abstract}
In light of today's massive data processing, digital computers are reaching fundamental performance limits due to physical limitations and energy consumption. For specific applications, tailored analog systems offer promising alternatives to digital processors. In this work, we investigate the potential of linear photonic chips for accelerating the computation of the singular value decomposition (SVD) of a matrix. The SVD is a key primitive in linear algebra and forms a crucial component of various modern data processing algorithms. Our main insights are twofold: first, hybrid systems of digital controller and photonic chip asymptotically perform on par with large-scale CPU/GPU systems in terms of runtime. Second, such hybrid systems clearly outperform digital systems in terms of energy consumption.

\end{abstract}

\section{Introduction}

Since the invention of transistors and integrated circuits around 1950, the processing power of digital computers has been exponentially increasing. This is due to the increasing number of transistors per integrated circuit (Moore's law) and a downscaling of transistor size (Dennard’s law). 
Today, fundamental limits in computing performance are being reached as the transistor scaling becomes less effective \cite{theis2017end}. Increasing data volume and computational models bring major challenges in the form of massive energy consumption for state-of-the-art computations.

Alternative computing concepts use well-controllable analog systems that occur in nature. 
The most prominent approach is based on analog electrical circuits where a set of passive and active components can be used to manipulate electric analog signals \cite{maclennan2007review}. Various other physical realizations are also possible, such as molecular computing \cite{conrad1990molecular} and fluid computing \cite{shanbhag2008search}. An optical analog system is a great physical medium for analog data processing as electromagnetic waves provide multiple degrees of freedom for information encoding and manipulation. Moreover, as the information manipulation in the coherent regime has no heat dissipation, it also provides a potentially orders-of-magnitude improvement in energy efficiency in certain cases.

In this work, we examine how key primitives in linear algebra can be solved efficiently via optical computation on photonic chips. For instance, Reck et al.\ \cite{reck1994experimental} showed that, for any unitary matrix $\mathbf U \in \C^{n\times n}$, there is an experiment, i.e., an assembly of phase-shifters and beamsplitters, which mimics the action of the matrix on $n$ optical channels. This is achieved by decomposing $\mathbf U = \mathbf U_1\cdots \mathbf U_m$ into a product of unitary matrices $\mathbf U_k \in \C^{n\times n}$, each encoding a two dimensional rotation embedded into $\C^n$. The action of $\mathbf U_k$ is reproduced by suitably configuring the corresponding phase-shifts in a Mach-Zehnder Interferometer (MZI), and $\mathbf U$ can be represented by composing several MZIs. For $\bz \in \C^n$ the matrix-vector product $\bU\bz$ can then be computed with ultra-low latency by encoding the coordinates of $\bz$ in the optical channels and passing them through the composition of MZIs. We view such an MZI configuration as a linear programmable optical chip that can be realized by a photonic integrated circuit (PIC).

By construction, such a PIC can encode only \emph{unitary} matrices $\bU \in \C^n$. To allow fast matrix-vector products with a general matrix $\bA \in \C^{m\times n}$, one could represent $\bA$ in terms of its singular value decomposition (SVD) $\bA = \bU \bSigma \bV^*$. The SVD consists of two unitary transforms $\bU \in \C^{m\times m}, \bV \in \C^{n\times n}$, and a diagonal matrix $\bSigma \in \R^{m \times n}$ the action of which can be realized by amplification or attenuation of the optical signals. 

Since obtaining an SVD of large matrices is costly on digital hardware, this leads to the question whether the SVD of $\bA$ can be computed more efficiently by using the optical hardware itself. Efficient analog SVD computation not only entails fast matrix-vector multiplications on such optical systems, but also fast solvers for linear systems since the (pseudo-)inverse of $\bA$ can be derived from its SVD as $\bA^{\dagger}=\bV\bSigma^{\dagger}\bU^*$. 

\subsection{Synopsis of our results}

We focus here on computing the SVD of a matrix $\bA \in \C^{n\times n}$ by a hybrid system that comprises a digital controller (DC) and a photonic integrated circuit (PIC) realizing unitary vector-matrix products in $\C^n$. The DC can perform digital operations and reprogram the unitary matrix $\bU$ encoded on the PIC, see Figure~\ref{fig:HybridSystem}.

\begin{figure}
    \centering
    \includegraphics[width=\linewidth]{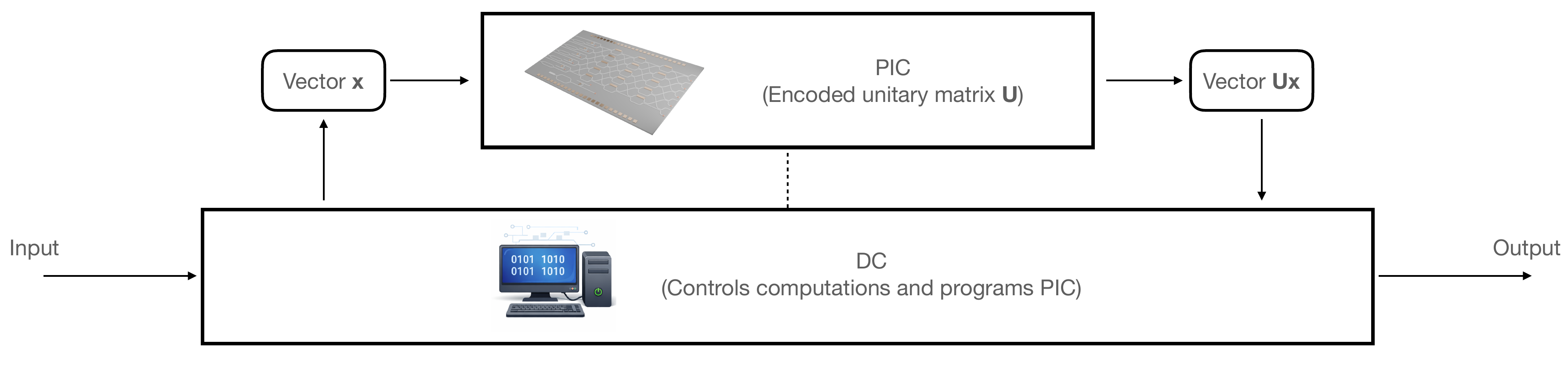}
    
    \caption{Schematic representation of the hybrid system comprising a digital controller (DC) and a photonic integrated circuit (PIC). Artwork for DC generated using ChatGPT.}
    \label{fig:HybridSystem}
\end{figure}

The time needed to evaluate $\bU\bz$ on the PIC, for some $\bz \in \C^n$, is governed by the time needed to encode the entries of $\bz$ into $n$ optical channels. It is small compared to the time needed to reprogram the PIC to a different unitary matrix $\bU'$, see Table \ref{tab:OPS-SAMARTH}. 
To use this hybrid system effectively for SVD algorithms, one has to well balance the number of matrix re-programmings against the number of matrix-vector products.

We study two algorithms for computing the SVD of a matrix $\bA \in \C^{m\times n}$. One naive method which proceeds by iteratively computing QR- and LQ-decompositions of $\bA$ in an alternating fashion, and one state-of-the-art method developed by Golub, Reinsch, and Kahan \cite{Golub1971,GolubKahanSVD} which combines an initial bidiagonalization step with fast chasing of the remaining off-diagonal entries along the diagonal. While the former method is a natural candidate for adaption to the hybrid system by mainly relying on computing orthogonal matrix products, the latter method has been heavily optimized in runtime complexity on digital systems and forms a benchmark for any SVD method.

Our results in Sections \ref{sec:ComputingSVD} and \ref{sec:Experiments} show that, due to slow convergence, neither digital nor hybrid implementation of the naive method can compete with a purely digital, parallelized implementation of the state-of-the-art method. However, we demonstrate that by supporting the bidiagonalization step of Golub-Reinsch-Kahan with a PIC, the resulting hybrid algorithm with single CPU and PIC reaches an asymptotic runtime on par with its digital counterpart that has access to an unlimited number of GPU cores. In addition, it clearly outperforms all digital implementations in terms of expected energy consumption.

\subsection{Related work}

Before we discuss our results in detail, let us review some of the relevant literature.\\

\textbf{SVD algorithms.} About 50 years ago, Golub, Reinsch and Kahan \cite{Golub1971,GolubKahanSVD} proposed an algorithm for computing the singular value decomposition (SVD) of a matrix $\bA$ that forms the basis for nearly all modern implementations. The matrix $\mathbf A$ is first bidiagonalized via Householder-reflections, followed by a variant of the QR-Algorithm to diagonalize the resulting matrix. The calculation of the singular values alone is reasonably fast, the construction of the unitary matrices $\bU$ and $\bV$ takes most of the computational effort. We will refer to this algorithm as GRK-SVD in the following.

An alternative and older approach to SVD computation is Jacobi's method \cite{jacobi1846leichtes}, which iteratively applies Givens rotations to $\bA$ until it has been diagonalized. While Jacobi's method is numerically more stable \cite{drmavc2008new}, it is slower than GRK-SVD and thus hardly used nowadays. There have been proposed several variations such as the cyclic Jacobi method \cite{forsythe1960cyclic} to accelerate it. Only recently, Drma\v{c} and Veseli\'c \cite{drmavc2008new,drmavc2008newII} proposed a sophisticated adaption of Jacobi's method that preserves its superior numerical stability and can compete with QR-based methods in terms of efficiency.\\

\textbf{Architecture of photonic integrated circuits.}
The decomposition algorithm one uses to obtain $\mathbf U_1,\dots,\mathbf U_m$ dictates the required configuration of the MZIs, i.e., the architecture of the photonic integrated circuit. 
In the case of Reck et al.\ \cite{reck1994experimental}, $n(n-1)/2$ MZIs have to be arranged in a triangular shape, see Figure \ref{fig:reck-architecture}, now commonly known as the Reck-architecture.

Building on the work of Reck et al., Clements et al.\ \cite{clements2017optimaldesignuniversalmultiport} present an improved rectangular architecture, achieving a higher density of components, see Figure \ref{fig:clements-architecture}. The number of MZIs needed is still $n(n-1)/2$. However, due to the shorter direct paths between the different MZIs, this architecture is more robust to optical phase mismatch. \\

\textbf{Further analog approaches to resource efficient linear algebra.} Aifer et al.~\cite{aifer2024thermodynamic} propose to solve basic tasks from linear algebra such as solving linear systems, inverting matrices, or computing matrix determinants by sampling from the equilibrium state of a suitably initialized thermodynamical system. They provide complexity comparison with state-of-the-art digital solvers of the respective tasks. Their approach is connected to digital Monte-Carlo sampling for solving linear algebra tasks \cite{forsythe1950matrix}.

Huang et al.~\cite{huang2016evaluation} examine analog accelerators for solving systems of linear equations that are based on ODE simulation. They conclude that such systems can lead to improvements in computation time and efficiency, but also highlight several of their challenges: accurate analog-to-digital conversion, saturation effects for variables with high dynamic range, and scalability of analog systems.

Finally, there has been growing interest in quantum computing architectures based on physical systems utilizing quantum superposition and entanglement. Their main advantage for computation relies on the development of tailored algorithms that can have an exponential speedup in time complexity over classical algorithms \cite{shor1999polynomial, montanaro2016quantum}. For instance, Harrow et al.~proposed a quantum algorithm for solving sparse and well-conditioned $n$-dimensional linear systems with complexity scaling as $\log(n)$ \cite{harrow2009quantum}.

However, the latter algorithm comes with massive resource requirements \cite{scherer2017concrete} which are critical in light of the difficulties of realizing large-scale quantum hardware \cite{biswas2017nasa,preskill2018quantum,de2021materials}. At the moment, the development of new quantum algorithms with proven computational advantage has made little progress. Recent works have even shown that for certain quantum algorithms there are classical algorithms with similar runtime \cite{tang2019quantum, tang2021quantum}.

\begin{figure}[t]
    \centering
    \begin{subfigure}[t]{0.47\textwidth}
        \centering
        \includegraphics[trim= 150 500 35 100, clip, width=0.9\linewidth]{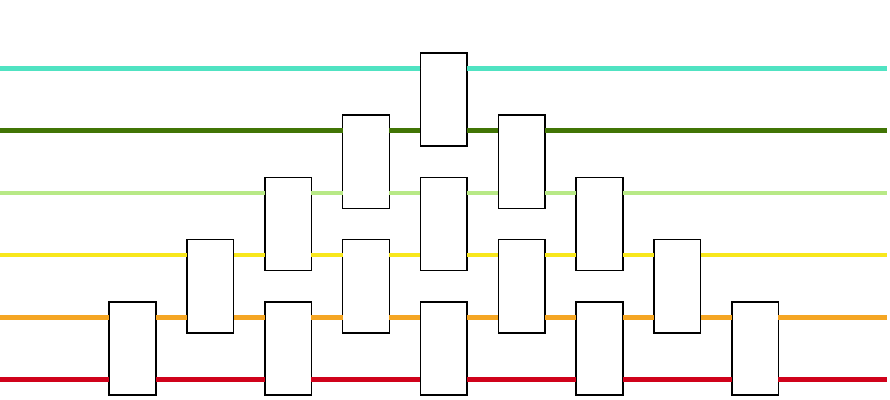}
        \caption{The Reck-Architecture}
        \label{fig:reck-architecture}
    \end{subfigure}
    ~ 
    \begin{subfigure}[t]{0.47\textwidth}
        \centering
        \includegraphics[trim= 150 500 35 100, clip, width=0.9\linewidth]{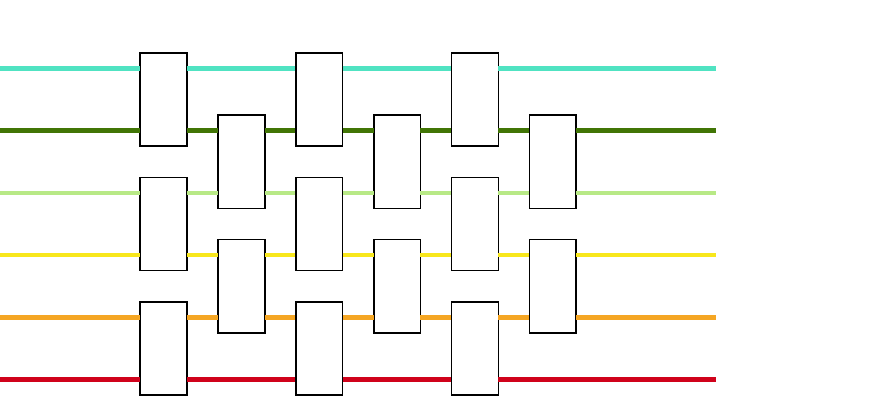}
        \caption{The Clements-Architecture}
        \label{fig:clements-architecture}
    \end{subfigure}
    \caption{MZI configurations of the photonic integrated circuit.}
\end{figure}

\subsection{Notation}
\label{sec:Notation}

In the following, we denote vectors and matrices in bold lower and upper case letters, i.e., a vector $\ba \in \R^d$ can be clearly distinguished from a matrix $\bA \in \R^{m\times d}$ with entries $\bA_{ij}$. 
We abbreviate the identity matrix of dimension $k$ by $\mathbf I_{k\times k }$.
For a vector $\bz \in \R^d$ and a matrix $\bA \in \R^{m\times d}$, we denote the (Euclidean) $\ell_2$-vector norm of $\bz$ by $\| \bz \|_2 = (\sum_{i=1}^d \bz_i^2)^{\frac{1}{2}}$ and the matrix $\ell_2$-norm of $\bA$ by $\| \bA \| = \sup_{\|\bz\|_2=1} \|\bA\bz \|_2$.

Finally, for $i < j$, let us denote by $\mathbf G_{i,j}^n(\theta) \in \R^{n\times n}$ the Givens rotation in $\R^n$ of angle $\theta$ along the plane spanned by $\be_i,\be_j \in \R^n$. Whenever clear from context, we will drop the superscript $n$. The matrix $\mathbf G_{i,j}^n(\theta) \in \R^{n\times n}$ differs from the identity matrix in only four positions: when restricted to the $i$-th and $j$-th rows and columns $\mathbf G_{i,j}^n(\theta)$ is
\begin{align*}
    \mathbf G_{i,j}^n(\theta) \Big|_{\{i,j\}\times\{i,j\}}
    = \begin{bmatrix}
        \cos(\theta) & -\sin(\theta) \\
        \sin(\theta) & \cos(\theta)
    \end{bmatrix},
\end{align*}
while $[\mathbf G_{i,j}^n(\theta)]_{kk} = 1$, for $k \notin \{i,j\}$, and $[\mathbf G_{i,j}^n(\theta)]_{kl} = 0$ for any $k,l \notin \{i,j\}$.
In our algorithms, we will often characterize $\theta$ by its tangent $r=\tan(\theta)$.
For brevity, we thus introduce the shorthand notation $\hat{\mathbf G}_{i,j}^n(r) := \mathbf G_{i,j}^n(\arctan(r))$ for which
\begin{align*}
    \hat{\mathbf G}_{i,j}^n(r) \Big|_{\{i,j\}\times\{i,j\}}
    = \begin{bmatrix}
        \cos(\arctan(r)) & -\sin(\arctan(r)) \\
        \sin(\arctan(r)) & \cos(\arctan(r))
    \end{bmatrix}
    = \frac{1}{\sqrt{1+r^2}}\begin{bmatrix}
        1 & -r \\
        r & 1
    \end{bmatrix}
\end{align*}

\section{Computing the SVD on the specified hybrid system}
\label{sec:ComputingSVD}

Existing routines for computing the SVD of a matrix $\bA \in \R^{m\times n}$ have been optimized for digital systems. Given our hybrid system (Figure \ref{fig:HybridSystem}), it is unclear which decomposition algorithm can optimally leverage the fast unitary matrix-vector products of PIC.

\subsection{SVD via alternating QR decomposition}
\label{sec:QR-SVD}

In a first step, we design a simple method for SVD computation that minimizes the number of DC operations, see Algorithm~\ref{alg:SVDviaQR}. We will refer to Algorithm~\ref{alg:SVDviaQR} as QR-SVD in the following. It approximates the SVD of $\bA$ by alternatingly applying QR- and LQ-decompositions, see Algorithm~\ref{alg:QR_Decomposition}. When implemented on our hybrid system, DC only needs to compute rank-1 matrices to program the required Householder transformations on PIC. 

\begin{algorithm}[t] 
	\caption{\textbf{:}  \textbf{SVDviaQR}} \label{alg:SVDviaQR}
	\begin{algorithmic}[1]
		\Require{$\mathbf A \in \mathbb R^{m\times n}$}
		\Statex
        \State $\mathbf U^{(0)} = \mathbf I_{m\times m}$
        \State $\bSigma^{(0)} = \mathbf A$
        \State $\mathbf V^{(0)} = \mathbf I_{n\times n}$
        \State Initialize $k=1$
        \State pick a $\delta > 0$ as the smallest number you want to consider as non-zero
        \While{$\lVert \bSigma^{(k)} - \mathrm{diag}(\boldsymbol{1}) \odot \bSigma^{(k)} \rVert_\infty > \delta$}
		      \State Decompose $(\bSigma^{(k)})^\top = \tilde{\mathbf Q} \tilde{\mathbf R}$ via Algorithm \ref{alg:QR_Decomposition}
            \State $\mathbf V^{(k+1)} = \tilde{\mathbf Q}^\top \cdot \mathbf V^{(k)}$
            \State $\bSigma^{(k+\frac{1}{2})} = \tilde {\mathbf R}^\top$
            \State Decompose $\bSigma^{(k+\frac{1}{2})} = \mathbf Q \mathbf R$ via Algorithm \ref{alg:QR_Decomposition}
            \State $\mathbf U^{(k)} = \mathbf U^{(k-1)} \cdot \mathbf Q $
            \State $\bSigma^{(k+1)} = \mathbf R$
		      \Let {$k$}{$k+1$}
		\EndWhile
		\Statex
		\Ensure{$\mathbf U = \mathbf U^{(k)}$ (unitary), $\mathbf V = (\mathbf V^{(k)})^T$ (unitary), and $\mathbf \bSigma = \bSigma^{(k)}$ (approximately diagonal) with $\mathbf A \approx \mathbf U \cdot \bSigma \cdot \mathbf V^T$}
	\end{algorithmic}
\end{algorithm}

While we did not find explicit references for QR-SVD, the method is related to the QR-algorithm for computing the eigenvalues of square matrices \cite{francis1961qr,kublanovskaya1962some}. The reason for this lack of references might be that due to the high digital costs of performing even a single QR-decomposition, QR-SVD is impractical on digital systems. Note that if QR-SVD converges, the limit is an SVD~of~$\bA$.

\begin{lemma}
    Let $\mathbf{A} \in \mathbb{R}^{m \times n}$ be an arbitrary matrix, and let $\bSigma^{(1)},\bSigma^{(\frac{3}{2})},\bSigma^{(2)},\dots$ be the sequence generated by Algorithm \ref{alg:SVDviaQR}.
    If $\bSigma^{(k/2)}$ converges, then for any $\varepsilon > 0$ there exists $k_\varepsilon \in \N$ such that, for all $k \ge k_\varepsilon$, there exists an SVD representation $\bA = \bU\bSigma\bV^T$ of $\bA$ with
    \begin{align*}
        \max \left\{ \| \bU^{(k/2)} - \bU \|_2, \| \bSigma^{(k/2)} - \bSigma \|_2, \| \bV^{(k/2)} - \bV \|_2 \right\} \le \varepsilon.
    \end{align*}
\end{lemma}

\begin{proof}

    First note that we have $\bA = \bU^{(k/2)} \cdot \bSigma^{(k/2)} \cdot (\bV^{(k/2)})^T$ for every $k$ by construction. Assume that $\bSigma^{(k/2)} \to \bSigma \in  \R^{m \times n}$. Clearly, there exists $k_\varepsilon \in \N$ such that $\| \bSigma^{(k/2)} - \bSigma \|_2 \le \varepsilon$, for all $k \ge k_\varepsilon$. Moreover, $\bSigma$ has to be a diagonal matrix since the iterates $\bSigma^{(k+\frac{1}{2})}$ are lower triangular while the iterates $\bSigma^{(k)}$ are upper triangular. 
    Define the set
    \begin{align*}
        \calU\calV_{\bSigma} = \{ (\bU,\bV) \in \R^{m\times m} \times \R^{n\times n} \colon \bU,\bV \text{ unitary and } \bA = \bU \bSigma \bV^T \}.
    \end{align*}
    Assume that there exists a subsequence $(\bU^{(k'/2)},\bV^{(k'/2)})$ such that $\mathrm{dist}((\bU^{(k'/2)},\bV^{(k'/2)}),\calU\calV_{\bSigma}) \ge \varepsilon$, for all $k'$. Then, there is a convergent subsequence $(\bU^{(k''/2)},\bV^{(k''/2)})$ converging to some $(\bU,\bV) \notin \calU\calV_{\bSigma}$ such that $\bU \bSigma \bV^T = \bA$. Contradiction. Hence, we conclude that $\mathrm{dist}((\bU^{(k'/2)},\bV^{(k'/2)}),\calU\calV_{\bSigma}) \to 0$ which verifies our claim. 
\end{proof}

The performance of QR-SVD on the hybrid system can be optimized by replacing the Householder reflections in Algorithm \ref{alg:QR_Decomposition} with Givens rotations. Indeed, when working with Householder reflections, DC has to compute their decompositions into products of two-dimensional rotations in order to encode them on PIC. This decomposition is cubic in the matrix size and should be avoided.

It can be bypassed by working with Givens rotations instead. To illustrate this, take $\mathbf a$ to be the $k$-th column of $\mathbf A^{(k-1)}$ below the diagonal, i.e., 
\begin{align*}
    \mathbf a = (\mathbf A^{(k-1)}_{k,k}, \mathbf A^{(k-1)}_{k+1,k}, \dots, \mathbf A^{(k-1)}_{m,k})^\top \in \mathbb R^{m-k+1},
\end{align*}
cf.\ Algorithm \ref{alg:QR_Decomposition}. In the $k$-th iteration, we need to construct a unitary transform that rotates $\ba$ to the first unit vector in $\mathbb R^{m-k+1}$. We will do so by replacing the Householder transform with a composition of $m-k$ Givens rotations. Recall the definition of Givens rotations $\mathbf G_{i, j}(\theta)$ and $\hat{\mathbf{G}}_{i, j}(r)$ in Section \ref{sec:Notation}.

We first apply a Givens-Rotation $\mathbf G_{m-1, m}(\theta_m)$ to $\mathbf a$ with $\theta_m$ chosen to zero its last entry. This can be computed digitally using a constant number of operations, as only two elements change. Repeating the process, we iteratively construct $\mathbf G_{\ell-1,\ell}(\theta_\ell)$, for $\ell = m-1,m-2,\dots,k+1$ to get $\mathbf G_{k, k+1}(\theta_{k+1}) \cdots \mathbf G_{m-1,m}(\theta_m) \ba = \| \ba \|_2 \cdot \be_1 \in \R^{m-k+1}$. This costs $\calO(m-k)$ digital operations.

Having computed the required rotation angles $\theta_{m},\dots,\theta_{k+1}$, we can now encode the respective 2D Givens rotations on the first $m-k$ blocks of the first diagonal of a PIC with Reck-architecture, see Figure \ref{fig:reck-architecture}.\footnote{If one designs a specialized PIC for QR-decomposition, it would suffice to use a strongly simplified architecture consisting of a single diagonal, cf.\ Figure \ref{fig:encoding_bidiag}.} 
Setting all other blocks to identity, PIC realizes the unitary transform $\mathbf G_{k, k+1}(\theta_{k+1}) \cdots \mathbf G_{m-1,m}(\theta_m)$ and can be used to compute $\bA^{(k)} = \mathbf G_{k, k+1}(\theta_{k+1}) \cdots \mathbf G_{m-1,m}(\theta_m) \cdot \bA^{(k-1)}$ and $\bQ^{(k)} =  \bQ^{(k-1)} \cdot \mathbf G_{m-1,m}^\top(\theta_m) \cdots \mathbf G_{k, k+1}^\top(\theta_{k+1}) $. To compute $\bQ^{(k)}$ with a PIC programmed to realize the multiplication with $\mathbf G_{k, k+1}(\theta_{k+1}) \cdots \mathbf G_{m-1,m}(\theta_m)$, we just pass $(\bQ^{(k-1)})^T$ through the PIC and transpose the result. Hence, the PIC doesn't need to be reprogrammed for the latter operation.

Our experiments in Section \ref{sec:Experiments} show that QR-SVD on the hybrid system cannot compete with state-of-the-art routines such as GRK-SVD on purely digital systems. The main efficiency bottleneck appears to be the slow convergence of QR-SVD. To reach high-precision solutions, the algorithm requires many iterations.

\begin{figure}[t]
    \centering
    \hspace{-2cm}
    \begin{subfigure}[t]{0.48\textwidth}
        \centering
        \includegraphics[trim= 0 490 140 0, clip, width=1.0\linewidth]{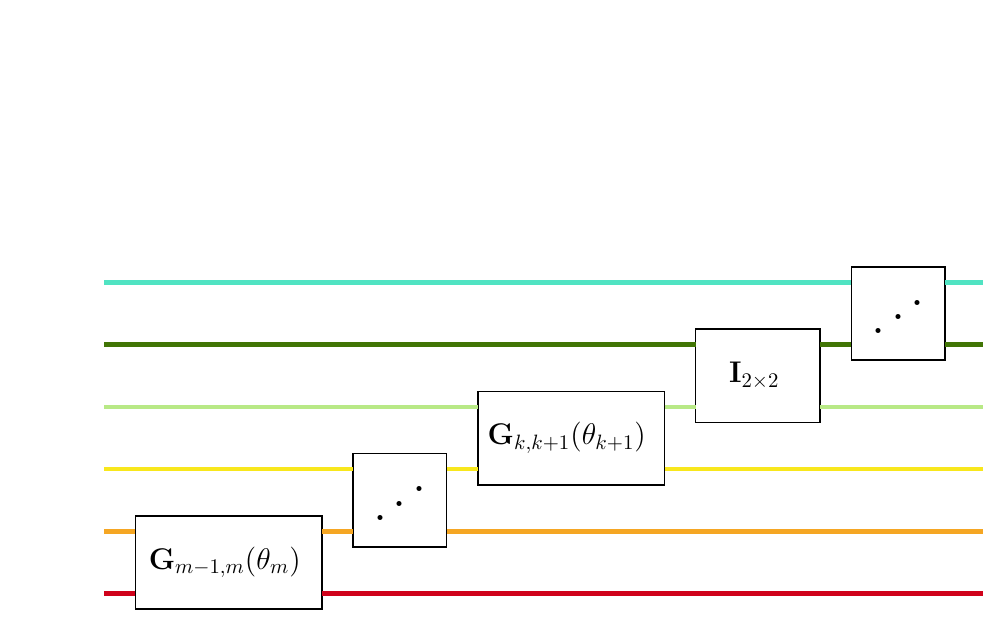}
        \caption{Example encoding for QR-Decomposition.}
        \label{fig:encoding_bidiag}
    \end{subfigure}
    ~ 
    \begin{subfigure}[t]{0.48\textwidth}
        \centering
        \includegraphics[trim= 0 492 140 0, clip, width=1.0\linewidth]{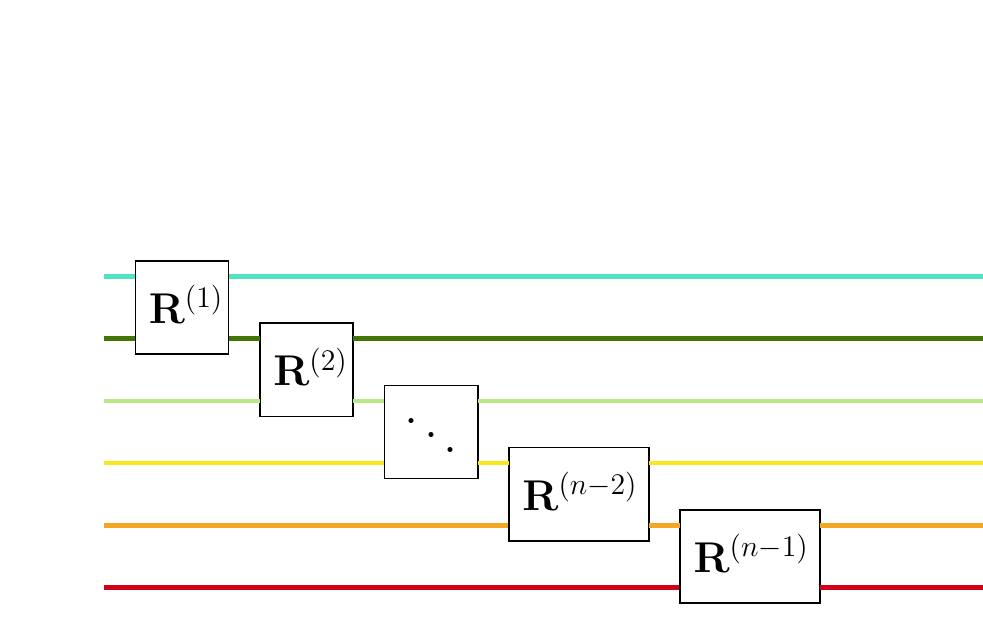}
        \caption{Example encoding for chasing phase, cf.\ Algorithm \ref{alg:rgkSVDforChip}.}
        \label{fig:encoding_qriteration}
    \end{subfigure}
    \caption{MZI configurations of the photonic integrated circuit for accelerating (i) the QR-decomposition in QR-SVD and the bidiagonalization step of GRK-SVD, and (ii) the composition of the singular vector matrices in the chasing phase of GRK-SVD. Note that (a) depicts the PIC configuration in the $k$-th iteration of Algorithm \ref{alg:PhotonicQRDecomposition}, and that flipping the PIC from (a) to (b) can be done without changing the hardware by reallocating the optical channels.}
\end{figure}

\subsection{Optimizing GRK-SVD for our hybrid system}

In light of the unfavorable performance of QR-SVD, we revisit the state-of-the-art algorithm GRK-SVD and examine how it can be optimized for our hybrid system. Let us assume in the following that $\bA \in \R^{m\times n}$ with $m \ge n$.
GRK-SVD consists of two main components: a bidiagonalization step and a subsequent chasing phase which removes the remaining off-diagonal entries.

In the bidiagonalization step, GRK-SVD applies $n$ Householder-reflections of decreasing size iteratively from each side to reduce $\bA$ to bidiagonal shape. The unitary transforms used in the $k$-th step are 
\begin{align*}
    \mathbf U^{(k)} =
    \begin{bmatrix}
       \mathbf I_{(k-1)\times(k-1)} & \boldsymbol 0 \\
         \boldsymbol 0 & \mathbf U_{\mathbf a - \| \mathbf a \|_2 \mathbf e_1}
    \end{bmatrix} \in \R^{m \times m}
    \qquad \text{and} \qquad
    \mathbf V^{(k)} =
    \begin{bmatrix}
       \mathbf I_{(k-1)\times(k-1)} & \boldsymbol 0 \\
         \boldsymbol 0 & \mathbf V_{\mathbf b - \| \mathbf b \|_2 \mathbf e_1}
    \end{bmatrix} \in \R^{n \times n},
\end{align*}
where $\ba$ and $\bb$ are the (partial) $k$-th column/row of $\bA$ that shall be rotated to a unitary vector, and $\mathbf U_{\mathbf z_U} = \mathbf I_{(m - k + 1)\times(m - k + 1)} - \frac{2}{\| \mathbf z_U \|_2^2} \mathbf z_U \mathbf z_U^\top$ and $\mathbf V_{\mathbf z_V} = \mathbf I_{(n - k + 1)\times(n - k + 1)} - \frac{2}{\| \mathbf z_V \|_2^2} \mathbf z_V \mathbf z_V^\top$, for $\mathbf z_U \in \mathbb R^{m-k+1}$, $\mathbf z_V \in \R^{n-k+1}$, and $\mathbf e_1$ denoting the first unit vector of the respective dimension.
Due to the required matrix-matrix products in each iteration, the bidiagonalization step heavily impacts the computational complexity of GRK-SVD, cf.\ Table \ref{tab:GRK-SVD-OPS-DSC}.

At this point, GRK-SVD has constructed a bidiagonal matrix $\bB \in \R^{m\times n}$ and unitary matrices $\bP \in \R^{m\times m}$, $\bQ \in \R^{n\times n}$ such that $\bA = \bP\bB\bQ$. To fully diagonalize $\bB$, GRK-SVD uses Givens rotations and chases the non-zero off-diagonal entries of $\bB$ along the second diagonal \cite{Golub1971}. By introducing Wilkinson shifts, this diagonalization method exhibits cubic convergence in general, cf.\ \cite{Golub1971}. Due to the bidiagonal shape of $\bB$, each multiplication with a Givens rotation can be computed quickly, even on a digital system. Combined with the cubic convergence, this leads to low computational complexity of the chasing step, see Table \ref{tab:GRK-SVD-OPS-DSC}.

perform the matrix products $\mathbf Q \mathbf U$ and $\mathbf V \mathbf P$ in order to get the full SVD.

When optimizing GRK-SVD for our hybrid system, only the bidiagonalization step needs to be changed substantially.\footnote{Indeed, applying Givens rotations to a bidiagonal matrix is of complexity $\calO(1)$ such that using PIC in the chasing phase would bring no relevant performance gain when computing the singular values. The only change in the hybrid implementation of the chasing phase is to directly build the two unitary transformation matrices containing the singular vectors on PIC. Instead of computing them as a product of several Givens rotations, one can encode the single Givens rotations in $2\times2$-blocks of PIC. Figure \ref{fig:encoding_qriteration} shows the exact arrangement of Givens-matrices on PIC, separately for the two steps of our algorithm.} 
Since this proceeds by interleaving the iterations of a QR- and an LQ-decomposition, we can slightly modify the scheme for accelerating the QR-decomposition of QR-SVD described in Section \ref{sec:QR-SVD}, see Algorithm \ref{alg:rgkSVDforChip}. The impact on computational complexity is illustrated in Table \ref{tab:GRK-SVD-OPS-Hybrid}.

\section{Experiments}
\label{sec:Experiments}

We will now present evidence for our main claims that (i) even a hybrid implementation of QR-SVD cannot compete with digital GRK-SVD, and that (ii) a hybrid implementation of GRK-SVD shows clear benefits when compared to its digital counterpart, both in terms of runtime and energy consumption.

To evaluate the effective runtime, we first subdivide both algorithms into basic operations and count the number of each operation. We consider three different versions of each algorithm: a purely digital implementation with sequential computing (\textbf{D}igital-\textbf{S}ingle\textbf{C}ore), a purely digital implementation in which matrix operations are executed on GPUs with an unlimited number of cores (\textbf{D}igital-\textbf{M}ulti\textbf{C}ore), and a hybrid implementation with single core CPU in which parts of the operations are executed on the PIC (\textbf{H}ybrid). The resulting operation counts for QR-SVD and GRK-SVD are summarized in Tables \ref{tab:QR-SVD-OPS}--\ref{tab:QR-SVD-OPS-Hybrid} and Tables \ref{tab:GRK-SVD-OPS-DSC}--\ref{tab:GRK-SVD-OPS-Hybrid}, respectively. Further details on how these counts have been derived are discussed in Appendix \ref{sec:OperationsCountsAlgorithm}.

We then fix a target precision of $10^{-6}$ for the off-diagonal entries of $\bSigma^{(k)}$ and estimate the expected number of iterations for each algorithm by linear regression over random executions. Figure~\ref{fig:ExpectedIterations} shows that the expected number of iterations can be well-described by a linear fit. Note that QR-SVD requires more iterations in general and shows a large variance of the necessary number of iterations. The worse convergence rate of QR-SVD is further illustrated in Figure \ref{fig:ConvergenceRate}. 

\begin{figure}
     \centering
     \begin{subfigure}[b]{0.45\textwidth}
         \centering
         \includegraphics[width=\textwidth]{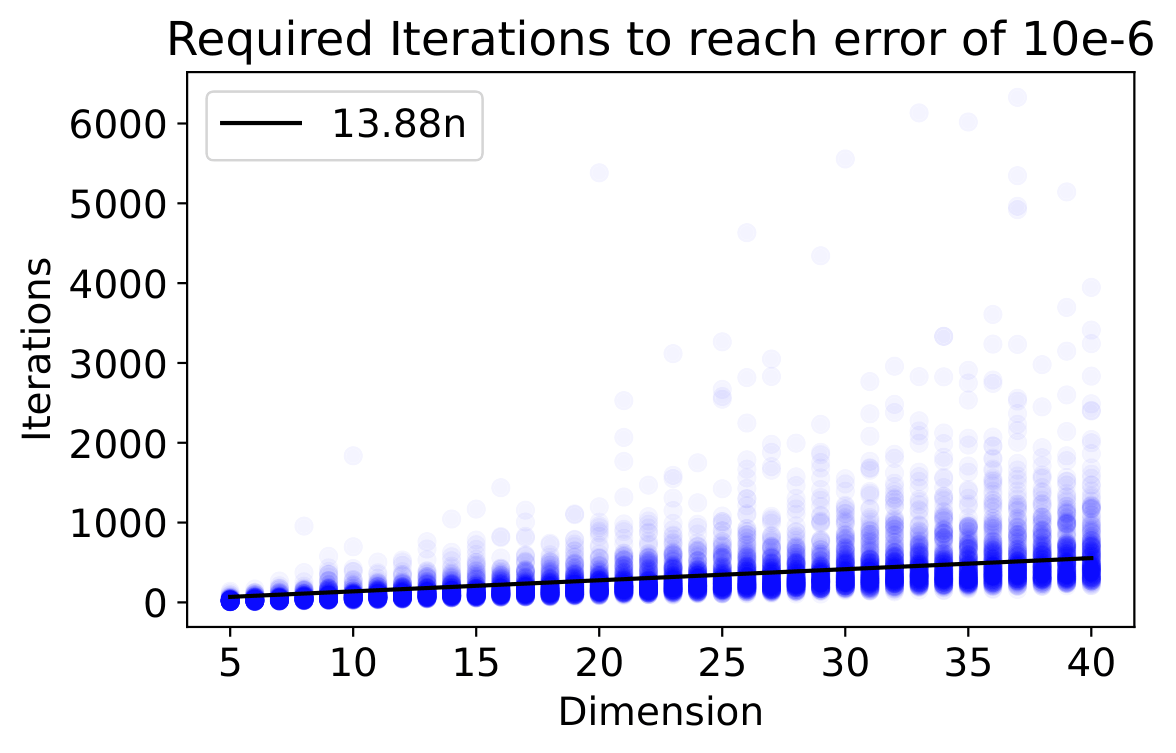}
         \caption{QR-SVD}
     \end{subfigure}
     \hfill
     \begin{subfigure}[b]{0.45\textwidth}
         \centering
         \includegraphics[width=\textwidth]{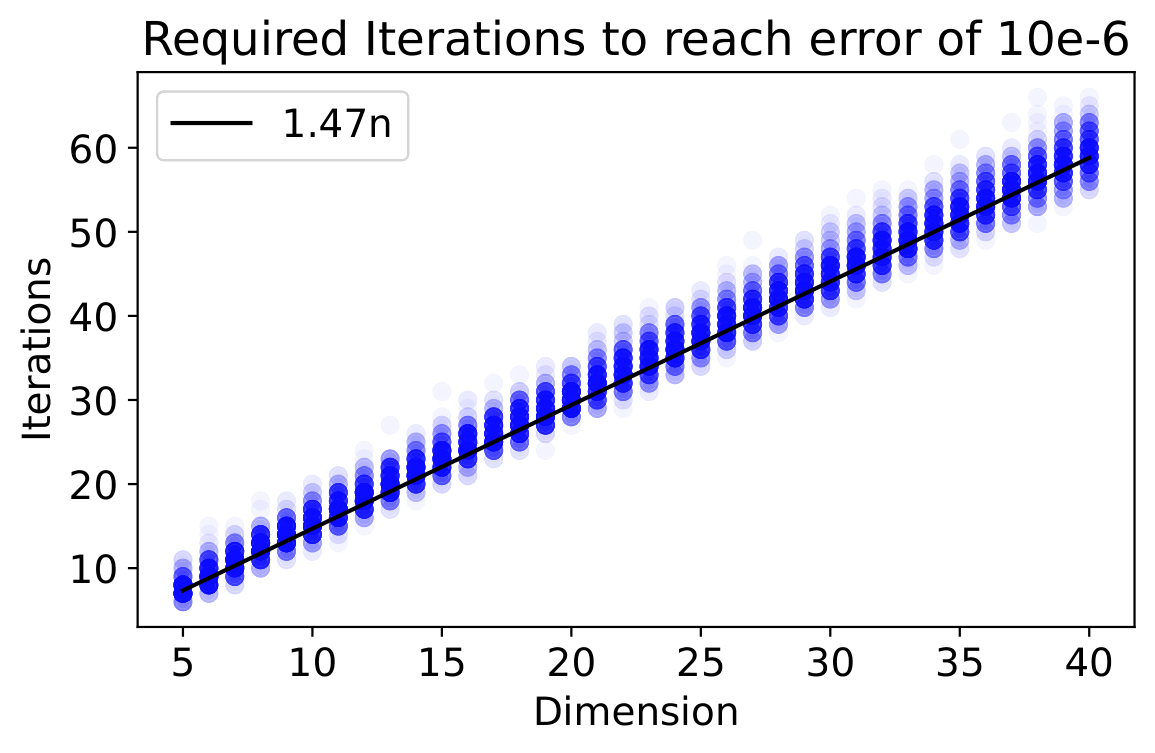}
         \caption{GRK-SVD}
     \end{subfigure}
        \caption{Number of iterations to reach target precision of $10^{-6}$ for QR-SVD and GRK-SVD. Linear best fit for medians is $\#\text{Iterations}=an+b$, where $a \approx 13.88$ and $b \approx -78.61$ (QR-SVD) and $a \approx 1.47$ and $b\approx 0.83$ (GRK-SVD), and $n$ denotes the dimension.}
        \label{fig:ExpectedIterations}
\end{figure}

\begin{figure}
     \centering
     \begin{subfigure}[b]{0.45\textwidth}
         \centering
         \includegraphics[width=\textwidth]{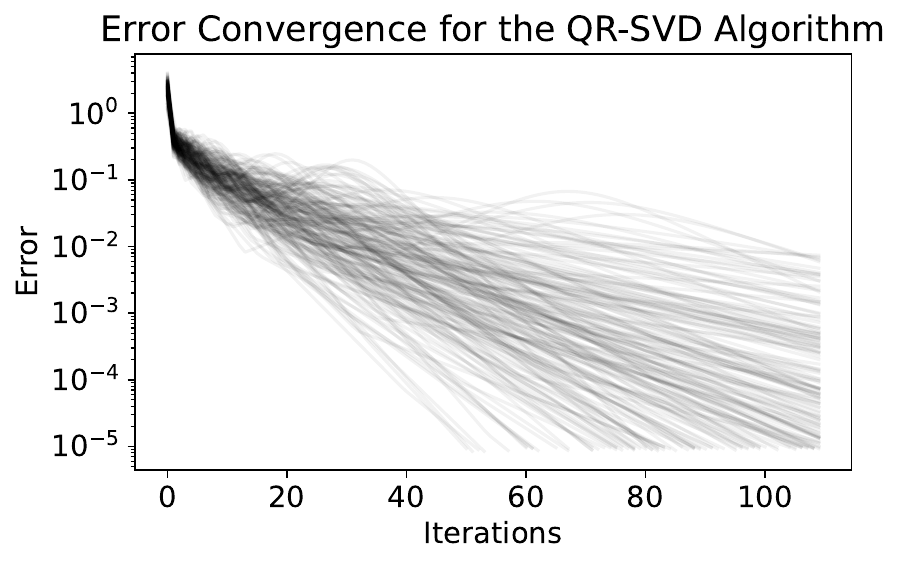}
         \caption{QR-SVD}
     \end{subfigure}
     \hfill
     \begin{subfigure}[b]{0.45\textwidth}
         \centering
         \includegraphics[width=\textwidth]{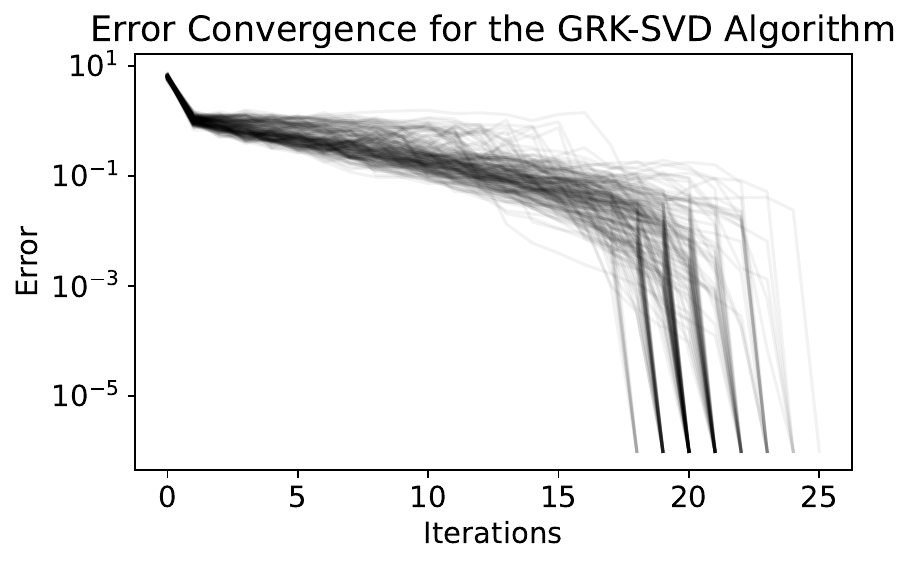}
         \caption{GRK-SVD}
     \end{subfigure}
        \caption{Error on off-diagonal entries of $\bSigma^{(k)}$ when applying QR-SVD and GRK-SVD to $200$ randomly drawn matrices in $\R^{15\times 15}$.} 
        \label{fig:ConvergenceRate}
\end{figure}

By combining the expected number of iterations of both algorithms in Figure \ref{fig:ExpectedIterations} with the counts of computational operations per method in Tables \ref{tab:QR-SVD-OPS}--\ref{tab:GRK-SVD-OPS-Hybrid} and multiplying them with time estimates for single CPU/GPU/PIC operations, see Table \ref{tab:OPS-COMPARISON}, we obtain runtime estimates for all described methods. The code to reproduce the presented experiments is provided at \url{https://github.com/Korbinian-Neuner/Computing-the-SVD-with-Photonic-Chips/tree/main}.\\

\begin{table}[t]
\centering
\begin{tabular}{|l|p{1cm} || l | p{1cm} || l | p{1cm}|}
  \hline
  \multicolumn{6}{|c|}{Relative runtime} \\
  \hline\hline
  CPU &  & GPU &  & PIC &  \\ \hline
  Addition & 1 & Addition & 4 & & \\
  Multiplication & 1 & Multiplication & 4 & & \\
  Division & 20 & & & & \\
  Square root & 15 & & & & \\
  & & & & PIC Configuration & 10000 \\
  & & & & PIC Operation & 50 \\
  \hline
\end{tabular}
\caption{Relative time cost estimates for different operations on our model systems \cite{vasilakis2015instruction}.
PIC Configuration refers to reprogramming the matrix $\bU$ encoded on PIC;
PIC Operation refers to computing a single matrix-vector product between $\bU$ and an input vector $\bv$.
The time estimates for PIC are measured relative to a 4GHz CPU and a GPU with clock rate 1GHz (0.25 ns represent one time unit here), and use the absolute time estimates in Table \ref{tab:OPS-SAMARTH} as reference, see Appendix \ref{sec:RelativeRuntime} for further details. We infer from the Agner-Fog instruction tables \cite{fog2018instruction}
that the time costs for divisions are slightly higher than for the square root.}
\label{tab:OPS-COMPARISON}
\end{table}

\begin{table}[t]
\centering
\begin{tabular}{|l|p{6cm}|}
    \hline
  Processor & total energy cost\\ \hline
  CPU (per time unit) & 375 pJ \\
  GPU (per FLOP) & 32.24 pJ \\
  PIC (Configuration) & $640n(n-1)$ pJ\\
  PIC (Operation) & $320n$ pJ \\
  \hline
\end{tabular}
\caption{Absolute energy cost estimates for one time unit on CPU, GPU, and different operations on a photonic integrated circuit. The PIC Operation values are estimated based on the state-of-the-art encoding/decoding costs in a transceiver while PIC Configuration costs are based on the energy costs of the MZI configuration via thermo-optic effect on a standard silicon PIC with $n$ channels. The CPU and GPU values are based on an AMD Ryzen 7 9800x3d consuming on average around 120W on 8 cores, which at 4GHz corresponds to 375 pJ per cycle per core, and an NVIDIA B200 consuming at peak around 1200W while reaching $37.22*10^{12}$ double flops, which amounts to around $32.24$ pJ per operation. These specifications can be found on the webpages of \href{https://www.amd.com/de/products/processors/desktops/ryzen/9000-series/amd-ryzen-7-9800x3d.html}{AMD} and \href{https://www.techpowerup.com/gpu-specs/b200.c4210}{NVIDIA}.}

\label{tab:OPS-ENERGY}
\end{table}

\textbf{Experiment 1 --- Runtime.} In the first experiment, we compare the expected runtime of QR-SVD and GRK-SVD in all types of implementation, i.e., digital single-core, digital multi-core, and hybrid (D-SC,D-MC,H). For simplicity, we only consider the SVD of square-matrices, i.e. $m=n$. Figure \ref{fig:ExpectedRuntime} shows that the hybrid implementation of QR-SVD is clearly outperformed by the digital implementations of GRK-SVD. This is as expected considering the slower convergence rate of QR-SVD. At the same time, we see that the hybrid implementation of GRK-SVD clearly outperforms the D-SC implementation of GRK-SVD and performs on par with the D-MC implementation that uses an unlimited number of GPU cores. \\

\begin{figure}
     \centering
     \begin{subfigure}[b]{0.45\textwidth}
         \centering
         \includegraphics[width=\textwidth]{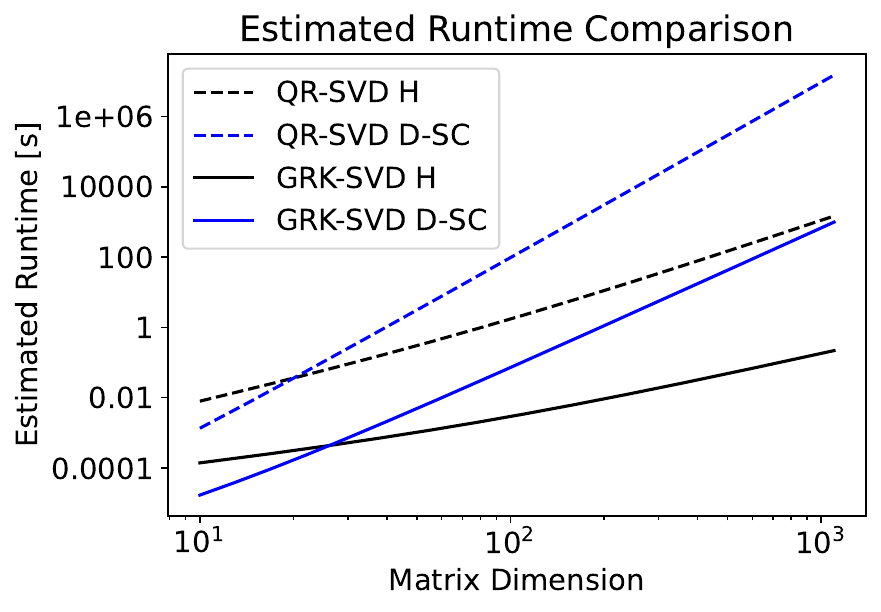}
         \caption{D-SC vs H}
     \end{subfigure}
     \hfill
     \begin{subfigure}[b]{0.45\textwidth}
         \centering
         \includegraphics[width=\textwidth]{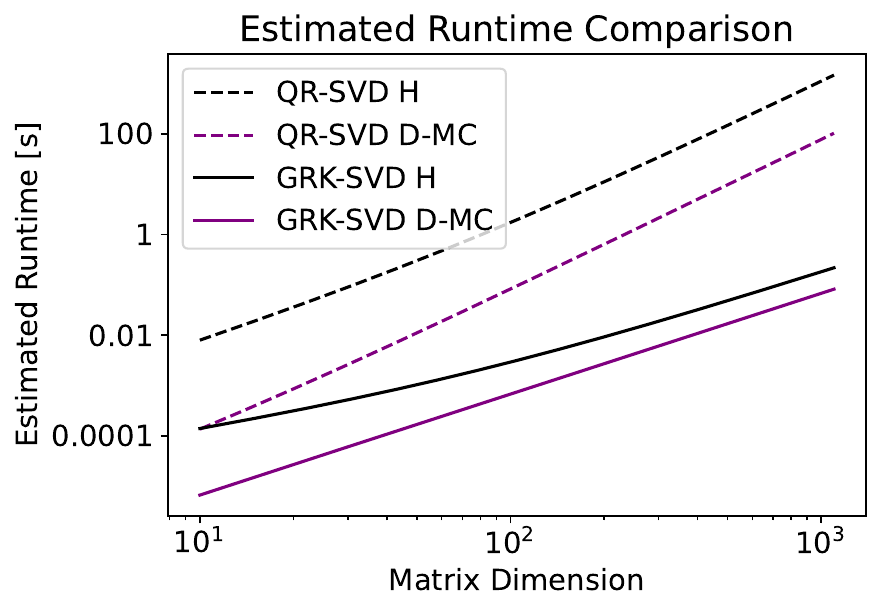}
         \caption{D-MC vs H}
     \end{subfigure}
        \caption{Comparison of expected runtime of QR-SVD (D-SC,D-MC,H) and GRK-SVD (D-SC,D-MC,H).}
        \label{fig:ExpectedRuntime}
\end{figure}

\textbf{Experiment 2 --- Energy consumption.} In the second experiment, we compare the expected energy consumption of QR-SVD and GRK-SVD in all types of implementation, i.e., digital single-core, digital multi-core, and hybrid (D-SC,D-MC,H). Again, we only consider the SVD of square-matrices. The energy consumption of the CPU is calculated by multiplying the number of operations executed on the CPU by the average energy cost per FLOP. This means that we treat the energy costs of different types of operations the same. We proceed analogously for the GPU where we use the counts of actual GPU operations, see Table \ref{tab:GRK-SVD-OPS-GPU}. The latter has been computed as the difference between the total number of operations (on CPU and GPU) and the parallelized count of the corresponding operation (on CPU). In case of the hybrid implementation, we add the expected energy consumption of the PIC operations, see Table \ref{tab:OPS-ENERGY}, to the expected energy consumption of the CPU operations. Figure~\ref{fig:ExpectedEnergy} illustrates that the hybrid implementation of GRK-SVD outperforms all digital implementations. A clear benefit can already be observed for matrices of size $n = \mathcal{O}(10^2)$, and the efficiency gap notably widens with increasing matrix size. \\

\begin{figure}
     \centering
     \begin{subfigure}[b]{0.45\textwidth}
         \centering
         \includegraphics[width=\textwidth]{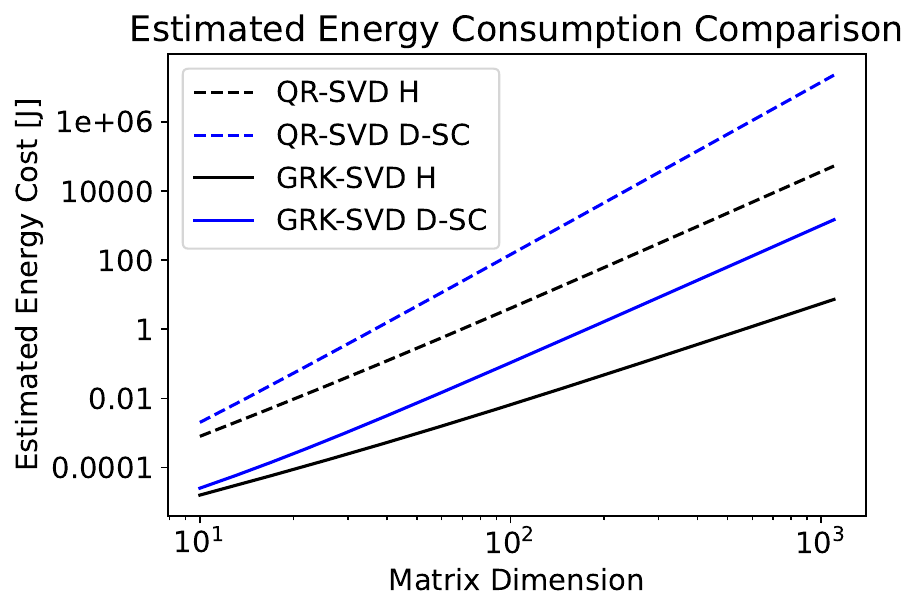}
         \caption{D-SC vs H}
     \end{subfigure}
     \hfill
     \begin{subfigure}[b]{0.45\textwidth}
         \centering
         \includegraphics[width=\textwidth]{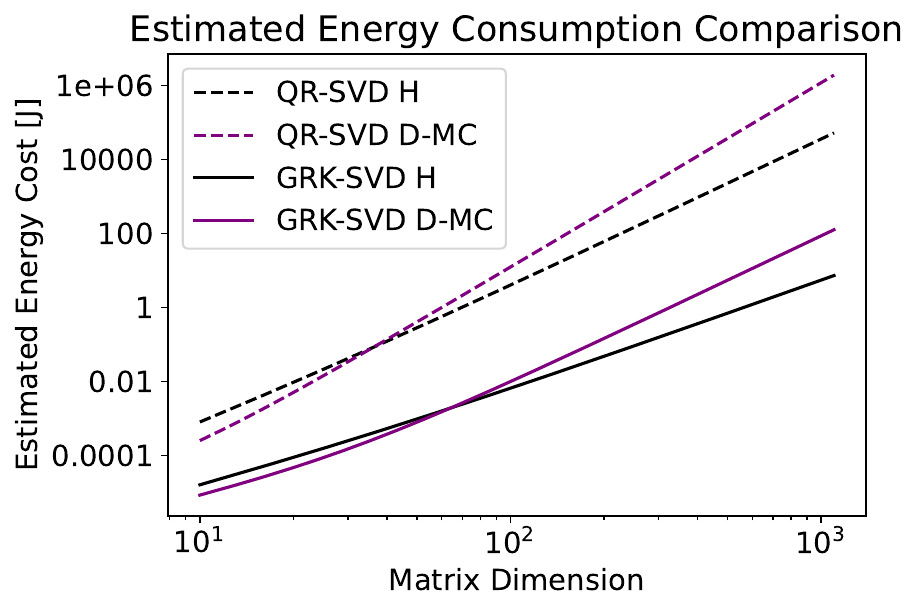}
         \caption{D-MC vs H}
     \end{subfigure}
        \caption{Comparison of expected energy consumption of QR-SVD (D-SC,D-MC,H) and GRK-SVD (D-SC,D-MC,H).}
        \label{fig:ExpectedEnergy}
\end{figure}

\textbf{On the validity of our comparison.} Figure \ref{fig:Timecheck} illustrates that our predicted runtimes match the actual runtime of QR-SVD and GRK-SVD. In the case of GRK-SVD, the actual runtime is about 6 times higher than the predicted runtime, which can be explained by the ignored memory operations to retrieve the non-cached data from a separate memory location. Omitting these memory retrieving operations is expected to have a higher impact on digital implementations compared to the hybrid implementations, as memory retrieval in hybrid system can be done in parallel to the PIC Operation. Thus, it does not weaken our claim that the hybrid implementation is competitive, even though we relied on conservative time estimates of PIC Operation and Configuration.

Finally, note that the implementations of QR-SVD and GRK-SVD in Algorithms \ref{alg:SVDviaQR} and \ref{alg:rgkSVD} are stated in a way to enhance conceptual clarity, not to optimize their actual runtime. Nevertheless, Figure \ref{fig:TimecheckLAPACK} shows that, as the matrix dimension grows, our predicted runtimes are a clear lower bound on the actual runtimes of optimized LAPACK routines. As before, this discrepancy can be explained by the ignored memory costs.

\begin{figure}
     \centering
     \begin{subfigure}[b]{0.45\textwidth}
         \centering
         \includegraphics[width=\textwidth]{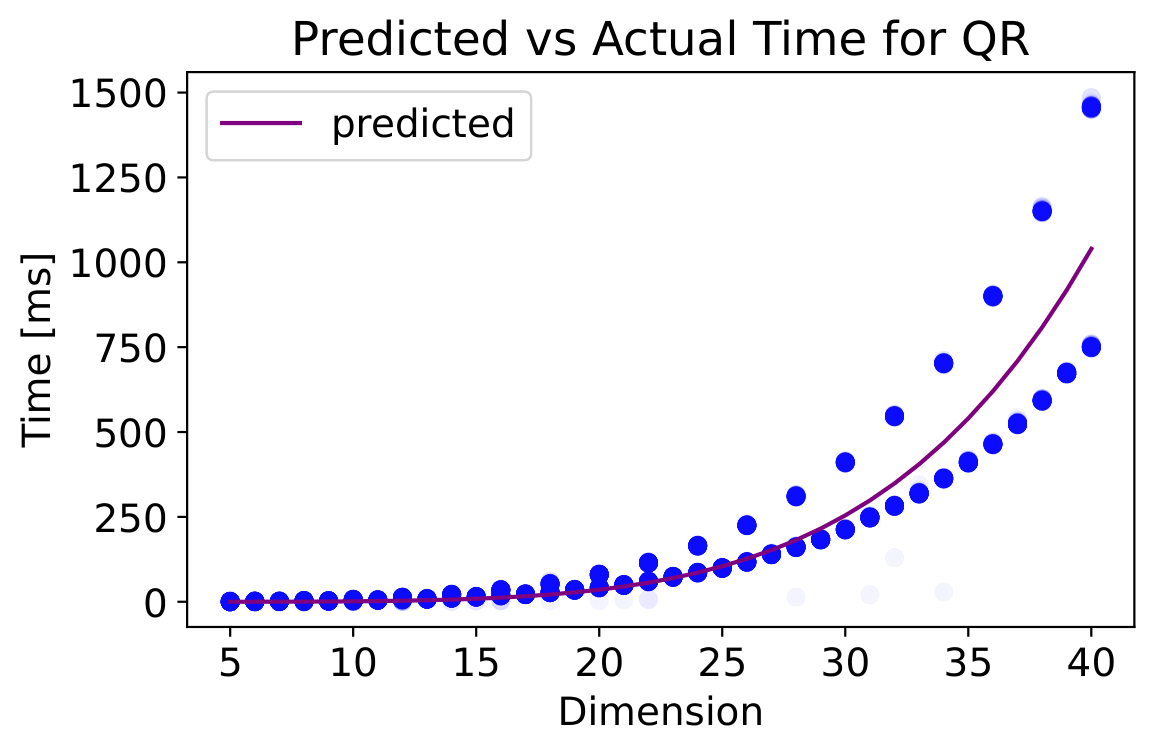}
         \caption{QR-SVD, D-SC}
     \end{subfigure}
     \hfill
     \begin{subfigure}[b]{0.45\textwidth}
         \centering
         \includegraphics[width=\textwidth]{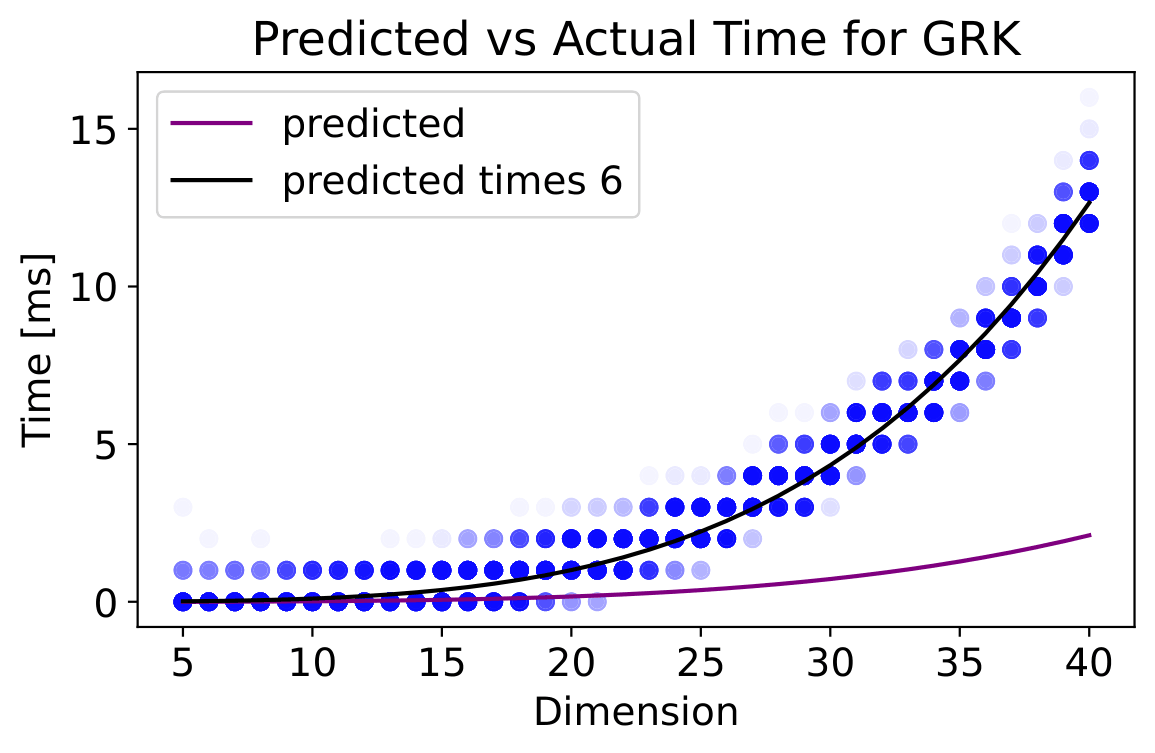}
         \caption{GRK-SVD, D-SC}
     \end{subfigure}
        \caption{Experimental verification of our time cost model comparing estimated and actual runtime of our digital implementations of QR-SVD and GRK-SVD.}
        \label{fig:Timecheck}
\end{figure}

\begin{figure}
     \centering
     \begin{subfigure}[b]{0.45\textwidth}
         \centering
         \includegraphics[width=\textwidth]{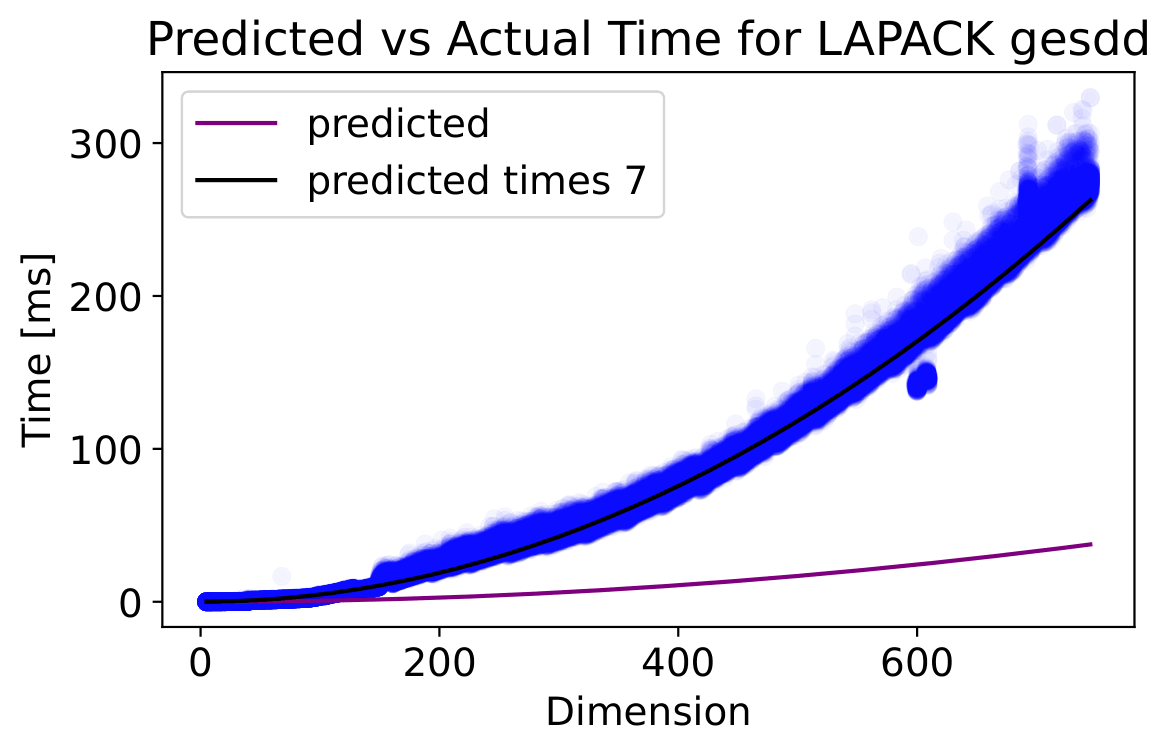}
         \caption{GRK-SVD D-MC vs LAPACK gesdd}
     \end{subfigure}
     \hfill
     \begin{subfigure}[b]{0.45\textwidth}
         \centering
         \includegraphics[width=\textwidth]{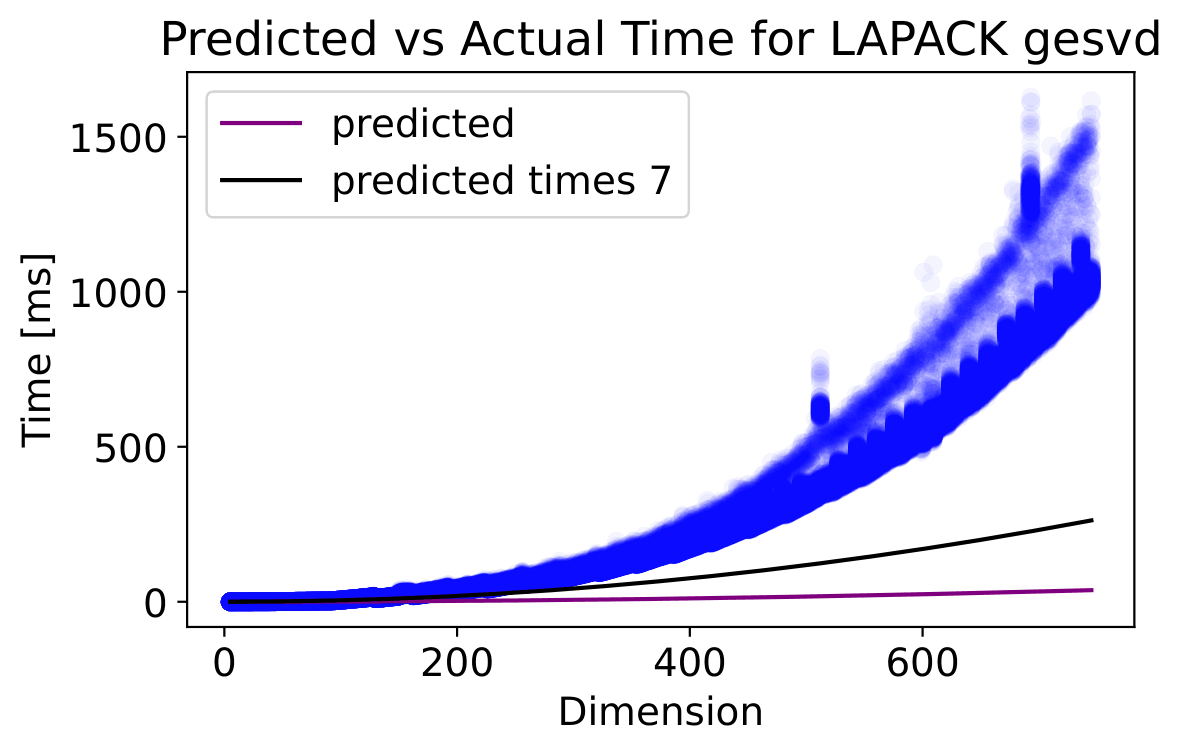}
         \caption{GRK-SVD D-MC vs LAPACK gesvd}
     \end{subfigure}
        \caption{Experimental verification of our time cost model against the optimized LAPACK routines gesdd and gesvd. Here we used the runtime estimates for the D-MC implementation.}
        \label{fig:TimecheckLAPACK}
\end{figure}

\section{Discussion}

The presented analysis shows that PICs yield promising solutions for increasing the energy efficiency of purely digital systems. We demonstrated that already for small matrix sizes of $n \approx 10^2$ a hybrid approach outperforms single core implementations in terms of runtime and even multi core implementations in terms of energy efficiency.

Computing the SVD is a core component of various modern data processing algorithms. Efficiency gains thus carry over to a multitude of methods including other linear algebraic primitives such as matrix inversion and solving linear systems. In future work, it would be desirable to identify further basic operations that can be outsourced to photonic components, and to design advanced hybrid methods for SVD computations and other important primitives, challenging digital state-of-the-art algorithms both in terms of runtime and energy efficiency. 

\section*{Author contribution statement}

JM and SV initiated the conducted research by joint discussions. JM carried out the main conceptual work and wrote great parts of the manuscript. KN optimized the hybrid implementations of the methods, carried out the operation counts, the experiments, and the computations of energy consumption, and provided detailed pseudocode of all involved methods. SV provided specifications of the discussed photonic integrated circuits. KN and SV helped finalizing the manuscript. 

\section*{Data and code availability statement}

All code to reproduce the experiments presented here can be accessed via \url{https://github.com/Korbinian-Neuner/Computing-the-SVD-with-Photonic-Chips/tree/main}. The datasets generated and/or analyzed during the current study are available here: \url{https://github.com/Korbinian-Neuner/Computing-the-SVD-with-Photonic-Chips/tree/main/Data}.



\newpage

\appendix
\section{Details on runtime evaluation}
\label{sec:OperationCountsAndRelativeTime}

\subsection{Operation counts of algorithms}
\label{sec:OperationsCountsAlgorithm}

We present here Tables \ref{tab:QR-SVD-OPS}--\ref{tab:GRK-SVD-OPS-Hybrid}. When counting the number of operations, we used the following conventions. 

We count subtraction and multiplication by $-1$ as single additions. We neither count operations to calculate indices nor memory operations to retrieve the non-cached data from a separate memory location. 
While the time costs for memory calls are relevant for all implementations, omitting them is favorable for the digital implementations since the hybrid system is expected to require less digital memory operations for performing calculations. 
Finally, we do not parallelize the multiplication of a bidiagonal matrix with a Givens rotation in the hybrid system since it only requires $4$--$6$ operations of the digital controller.

Due to the travelling time of light there is a small dependence of the time estimates in Table \ref{tab:OPS-SAMARTH} on the size of the PIC. We implicitly fixed the PIC length in our time estimates.

Finally, we assume that the digital system uses a single-core CPU and a GPU with potentially infinitely many cores, i.e., one GPU operation corresponds to an arbitrary number of elementary operations that are fully parallelizable. In contrast, the hybrid system uses a single-core CPU and a single PIC of the respective input dimension $n$.

\begin{table}[t]
\centering
\begin{tabular}{|l|p{4cm}|p{4cm}|p{4cm}|}
\hline
  Operation & Count (total) & Count (parallelized) & actual GPU count \\ \hline
  Addition & $(m^3n +m^3 +2m^2n -\frac{1}{3}mn^3 +\frac{19}{3}mn -m +\frac{2}{3}n^4 +\frac{5}{3}n^3 +\frac{1}{3}n^2 +\frac{19}{3}n -7 )C$ & $(3mn +5n -4 )C$ & $(m^3n +m^3 +2m^2n -\frac{1}{3}mn^3 +\frac{10}{3}mn -m +\frac{2}{3}n^4 +\frac{5}{3}n^3 +\frac{1}{3}n^2 +\frac{4}{3}n -3 )C$\\
  Addition (GPU) & - & $(2mn +m +9n -5 )C$ & - \\
  Multiplication & $(m^3n +m^3 +3m^2n -\frac{1}{3}mn^3 -mn^2 +\frac{16}{3}mn -m +\frac{2}{3}n^4 +\frac{7}{3}n^3 +\frac{1}{3}n^2 +\frac{8}{3}n -5 )C$ & $(3mn +3n -3 )C$ & $(m^3n +m^3 +3m^2n -\frac{1}{3}mn^3 -mn^2 +\frac{7}{3}mn -m +\frac{2}{3}n^4 +\frac{7}{3}n^3 +\frac{1}{3}n^2 -\frac{1}{3}n -2 )C$ \\
  Multiplication (GPU) & - & $(2mn +m +7n -4 )C$ & -\\
  Division & $(2n -1 )C$ & $(2n -1 )C$ & - \\
  Square root & $(4n -2 )C$ & $(4n -2 )C$ & - \\
  \hline
\end{tabular}
\caption{Operation counts for QR-SVD. Total count corresponds to QR-SVD (D-SC), parallelized count corresponds to QR-SVD (D-MC). The actual GPU count shows how many of the corresponding operation are actually performed on the GPU in the D-MC implementation, which is relevant when estimating the energy consumption.}
\label{tab:QR-SVD-OPS}
\end{table}

\begin{table}[t]
\centering
\begin{tabular}{|l|p{5cm}|}
\hline
  Operation & Count  \\ \hline
  Addition & $(5mn-5n)C$ \\
  Multiplication & $(4mn-4n)C$ \\
  Division & $(3mn-3n)C$ \\
  Square root & $(2mn-2n)C$ \\
  PIC Configuration & $(m+n)C$ \\
  PIC Operation & $(m^2 +2mn +n^2 )C$ \\
  \hline
\end{tabular}
\caption{Operation counts for QR-SVD (H).}
\label{tab:QR-SVD-OPS-Hybrid}
\end{table}


\begin{table}[t]
\centering
\begin{tabular}{|l|p{5cm}|p{6cm}|}
\hline
  Operation & Count in Bidiag-step & Count in QR-Iter/Chasing \\ \hline
  Addition & $m^3n +2m^2n -\frac{1}{3}mn^3 -mn^2 +\frac{19}{3}mn -m +\frac{2}{3}n^4 -\frac{1}{3}n^3 -\frac{2}{3}n^2 +\frac{4}{3}n -8 $ & $(2n-2)mC + (2n^2+17n-16)C$ \\
  Multiplication & $m^3n +3m^2n -\frac{1}{3}mn^3 -2mn^2 +\frac{16}{3}mn -m +\frac{2}{3}n^4 +\frac{1}{3}n^3 -\frac{5}{3}n^2 -\frac{1}{3}n -5$ & $(4n-4)mC + (4n^2+24n-27)C$ \\
  Division & $2n-2$ & $(4n-3)C$ \\
  Square root & $4n-4$ & $(2n-1)C$ \\
  \hline
\end{tabular}
\caption{Operation counts for GRK-SVD (D-SC).}
\label{tab:GRK-SVD-OPS-DSC}
\end{table}

\begin{table}[t]
\centering
\begin{tabular}{|l|p{5cm}|p{6cm}|}
\hline
  Operation & Count in Bidiag-step & Count in QR-Iter/Chasing \\ \hline
  Addition & $3mn+2n-5$ & $7nC$ \\
  Addition (GPU) & $2mn + 6n -8$ & $(4n-4)C$ \\
  Multiplication & $3nm-3$ & $(4n +5 )C$ \\
  Multiplication (GPU) & $2mn +4n -6$ & $(8n -8 )C$ \\
  Division & $2n-2$ & $(4n-3)C$ \\
  Square root & $4n-4$ & $(2n-1)C$ \\
  \hline
\end{tabular}
\caption{Operation counts for GRK-SVD (D-MC).}
\label{tab:GRK-SVD-OPS-DMC}
\end{table}

\begin{table}[t]
\centering
\begin{tabular}{|l|p{5cm}|p{6cm}|}
\hline
  Operation & Count in Bidiag-step & Count in QR-Iter/Chasing \\ \hline
  Addition & $5mn -10n +5$ & $(18n -16 )C$ \\
  Multiplication & $4mn -8n +4 $ & $(28n -27 )C$ \\
  Division & $3nm-6n+3$ & $(4n-3)C$ \\
  Square root & $2nm-4n+2$ & $(2n-1)C$ \\
  PIC Configuration & $2n$ & $2C$ \\
  PIC Operation & $2mn+2n^2$ & $(m+n)C$ \\
  \hline
\end{tabular}
\caption{Operation counts for GRK-SVD (H).}
\label{tab:GRK-SVD-OPS-Hybrid}
\end{table}

\begin{table}[t]
\centering
\begin{tabular}{|l|p{5cm}|p{6cm}|}
\hline
  Operation & actual GPU count in Bidiag-step & actual GPU count in QR-Iter/Chasing \\ \hline
  Addition & $m^3n +2m^2n -\frac{1}{3}mn^3 -mn^2 +\frac{10}{3}mn -m +\frac{2}{3}n^4 -\frac{1}{3}n^3 -\frac{2}{3}n^2 -\frac{2}{3}n -3 $ & $(2n-2)mC + (2n^2+10n-16)C$ \\
  Multiplication & $m^3n +3m^2n -\frac{1}{3}mn^3 -2mn^2 +\frac{7}{3}mn -m +\frac{2}{3}n^4 +\frac{1}{3}n^3 -\frac{5}{3}n^2 -\frac{1}{3}n -2$ & $(4n-4)mC + (4n^2+20n-32)C$ \\
  \hline
\end{tabular}
\caption{Operations actually performed by the GPU for GRK-SVD (D-MC).}
\label{tab:GRK-SVD-OPS-GPU}
\end{table}

\subsection{Computation of relative runtime for PIC}
\label{sec:RelativeRuntime}

To estimate the relative runtime of PIC Configuration and PIC Operation, we use the absolute time estimates provided in Table \ref{tab:OPS-SAMARTH}. PIC Operation, i.e., computing a single matrix-vector product between $\bU$ encoded on PIC and an input vector $\bv$, consists of one encoding step, one decoding step and the data throughput time needed for light to pass through the circuit. In our experiments, we use the average times for operations to calculate the total runtime. For comparison, Table \ref{tab:OPS-SAMARTH} also provides more optimistic time estimates that should be achievable in the future. 

\begin{table}[t]
\centering
\begin{tabular}{|l|p{3cm}|p{2cm}|}
\hline
  Operation & time (average) & time (fast) \\ \hline
  PIC Configuration & 2.5 ms & 1 ns \\
  PIC Encoding/Decoding & 1 ns & 10 ps \\
  PIC Operation & 12.5 ns & 1ns \\
  \hline
\end{tabular}
\caption{Time cost estimates for different operations on a photonic integrated circuit. Configuration time refers to programming a new matrix $\bU$ on PIC. Operation time covers the time data needs for passing through the entire PIC, including encoding, time-of-flight latency and decoding. Average times are conservative estimates based on thermo-optic configuration and electro-optic encoding while fast time estimations are based on the state-of-the-art PIC implementations.
}

\label{tab:OPS-SAMARTH}
\end{table}

\section{Pseudocode for all presented algorithms}

We present here the detailed pseudocode for all algorithms discussed in Sections \ref{sec:ComputingSVD} and \ref{sec:Experiments}:
\begin{itemize}
    \item The basic QR-decomposition (Algorithm \ref{alg:QR_Decomposition})
    \item The GRK-SVD bidiagonalization step (Algorithm \ref{alg:rgkBDG})
    \item The full GRK-SVD algorithm (Algorithm \ref{alg:rgkSVD})
    \item The QR-decomposition optimized for the hybrid system (Algorithm \ref{alg:PhotonicQRDecomposition})
    \item The GRK-SVD bidiagonalization step optimized for the hybrid system (Algorithm \ref{alg:rgkSVDforChip})
    \item The full GRK-SVD algorithm optimized for the hybrid system (Algorithm \ref{alg:PhoRGKSVD})
\end{itemize}

\begin{algorithm}[t] 
	\caption{\textbf{:}  \textbf{QR-decomposition}} \label{alg:QR_Decomposition}
	\begin{algorithmic}[1]
		\Require{$\mathbf A = \mathbf A^{(0)} \in \mathbb R^{m\times n}$}
		\Statex
        \State $\mathbf Q^{(0)} = \mathbf I_{m\times m}$
        \State Initialize $k=1$
        \While{$k \le \min\{m,n\}$}
		\State Define $\mathbf a = (A^{(k-1)}_{k,k}, A^{(k-1)}_{k+1,k}, \dots, A^{(k-1)}_{m,k})^\top \in \mathbb R^{m-k+1}$, i.e., $\mathbf a$ is the $k$-th column of $\mathbf A^{(k-1)}$ below the diagonal 
        \State Define the unitary matrix
        \begin{align*}
            \mathbf U^{(k)} =
            \begin{bmatrix}
                \mathbf I_{(k-1) \times (k-1)} & \boldsymbol 0 \\
                \boldsymbol 0 & \mathbf U_{\mathbf a - \| \mathbf a \|_2 \mathbf e_1}
            \end{bmatrix},
        \end{align*}
        where $\mathbf U_{\mathbf v} = \mathbf I_{(m - k + 1) \times (m - k + 1)} - \frac{2}{\| \mathbf v \|_2^2} \mathbf v \mathbf v^\top$, for $\mathbf v \in \mathbb R^d$, and $\mathbf e_1 \in \mathbb R^{m-k+1}$ denotes the first unit vector.
        \State $\mathbf A^{(k)} = \mathbf U^{(k)} \cdot \mathbf A^{(k-1)}$
        \State $\mathbf Q^{(k)} = \mathbf Q^{(k-1)} \cdot \mathbf U^{(k)}$ 
		\Let {$k$}{$k+1$}
		\EndWhile
		\Statex
		\Ensure{$\mathbf R = \mathbf A^{(\min\{m,n\})}$ (upper triangular) and $\mathbf Q = \mathbf Q^{(\min\{m,n\})}$ (unitary) with $\mathbf A = \mathbf Q \cdot \mathbf R$}
	\end{algorithmic}

\end{algorithm}

\begin{algorithm}[t]
    \caption{\textbf{:}  \textbf{Golub-Reinsch-Kahan-Bidiagonalization}}
    \label{alg:rgkBDG}
    \begin{algorithmic}[1]
        \Require{$\mathbf A \in \mathbb R^{m\times n}$ with $m \geq n$}
		\Statex
        \State $\mathbf B^{(0)} = \mathbf A$
        \State $\mathbf P^{(0)} = \mathbf I_{m\times m}$
        \State $\mathbf Q^{(0)} = \mathbf I_{n\times n}$
        \State Initialize $k=1$
        \While{$k \le n$}
        \State Define $\mathbf a = (\mathbf B^{(k-1)}_{k,k}, \mathbf B^{(k-1)}_{k+1,k}, \dots, \mathbf B^{(k-1)}_{m,k})^\top \in \mathbb R^{m-k+1}$, i.e., $\mathbf a$ is the $k$-th column of $\mathbf B^{(k-1)}$ below the diagonal, and the unitary matrix
        \begin{align*}
            \mathbf U^{(k)} =
            \begin{bmatrix}
                \mathbf I_{(k-1) \times (k-1)} & \boldsymbol 0 \\
                \boldsymbol 0 & \mathbf U_{\mathbf a - \| \mathbf a \|_2 \mathbf e_1}
            \end{bmatrix},
        \end{align*}
        where $\mathbf U_{\mathbf v} = \mathbf I_{(m - k + 1)\times (m - k + 1)} - \frac{2}{\| \mathbf v \|_2^2} \mathbf v \mathbf v^\top$, for $\mathbf v \in \mathbb R^d$, and $\mathbf e_1 \in \mathbb R^{m-k+1}$ denotes the first unit vector.
        
        \State $\mathbf{\tilde{B}}^{(k)} = \mathbf U^{(k)} \cdot \mathbf B^{(k-1)}$
        \State $\mathbf P^{(k)} = \mathbf P^{(k-1)} \cdot \mathbf U^{(k)}$ 
        
        \State Define $\mathbf b = (B^{(k-1)}_{k,k + 1}, B^{(k-1)}_{k,k + 2}, \dots, B^{(k-1)}_{k,n})^\top \in \mathbb R^{n - k}$, i.e., $\mathbf b$ is the $k$-th row of $\mathbf B^{(k-1)}$ to the right of the superdiagonal. If $k=n$, set $\mathbf V^{(k)} = \mathbf I_{n\times n}$ and skip the calculations. Define the unitary matrix
        \begin{align*}
            \mathbf V^{(k)} =
            \begin{bmatrix}
                \mathbf I_{k \times k} & \boldsymbol 0 \\
                \boldsymbol 0 & \mathbf V_{\mathbf b - \| \mathbf b \|_2 \mathbf e_1}
            \end{bmatrix},
        \end{align*}
        where $\mathbf V_{\mathbf v} = \mathbf I_{(n-k)\times (n-k)} - \frac{2}{\| \mathbf v \|_2^2} \mathbf v \mathbf v^\top$, for $\mathbf v \in \mathbb R^{n-k}$, and $\mathbf e_1 \in \mathbb R^{n-k}$ denotes the first unit vector.
        
        \State $\mathbf B^{(k)} = \mathbf{\tilde{B}}^{(k)} \cdot \mathbf V^{(k)}$
        \State $\mathbf Q^{(k)} = \mathbf V^{(k)} \cdot \mathbf Q^{(k-1)}$ 
		\Let {$k$}{$k+1$}
		\EndWhile
		\Statex
		\Ensure{$\mathbf B = \mathbf B^{(n)}$ (bidiagonal), $\mathbf P = \mathbf P^{(n)}$ (unitary) and $\mathbf Q = \mathbf Q^{(n)}$ (unitary) with $\mathbf A = \mathbf P \cdot \mathbf B \cdot \mathbf Q$}
    \end{algorithmic}
    
\end{algorithm}

\begin{algorithm}[t]
    \caption{\textbf{:}  \textbf{GRK-SVD (Golub, Reinsch, Kahan)}}
    \label{alg:rgkSVD}
    \begin{algorithmic}[1]
        \Require{$\mathbf A \in \mathbb R^{m\times n}$ with $m \geq n$}
		\Statex
        \State Decompose $\mathbf A = \mathbf P \cdot \mathbf B \cdot \mathbf Q$ via Algorithm \ref{alg:rgkBDG}
        \State Set $\mathbf U^{(0)} = \mathbf P$
        \State Set $\mathbf V^{(0)} = \mathbf Q$
        \State Set $\delta = \epsilon_0 \lVert \mathbf B\rVert_\infty$ where $\epsilon_0$ denotes machine precision
        \While{$\lVert \mathbf B - $diag($\mathbf B$)$ \rVert_\infty > \delta$}
            \State Set the Wilkinson Shift $s$ as the dominant eigenvalue of the bottom right 2x2 minor of $\mathbf B^T \mathbf B$:
            \begin{align*}
                s = \max(\sigma(\mathbf{M})), \quad \text{for} \quad \mathbf M =
                \begin{bmatrix}
                    (\mathbf B_{n-2, n-1})^2+(\mathbf B_{n-1, n-1})^2   & \mathbf B_{n-1, n-1}\mathbf B_{n-1, n}   \\
                    \mathbf B_{n-1, n-1}\mathbf B_{n-1, n} & (\mathbf B_{n-1, n})^2+(\mathbf B_{n, n})^2
                \end{bmatrix},
            \end{align*}
            \State Set the desired ratio $r_1$ of the first two entries in the first column of $\mathbf B^T \mathbf B - s \mathbf I_{n\times n}$ to
            \begin{align*}
                r_1 = \frac{\mathbf B_{1,2}\mathbf B_{1,1}}{(\mathbf B_{1,1})^2-s}
            \end{align*}
            \State Construct the first Givens-Rotation $\mathbf R^{(1)}$ as $\mathbf R^{(1)} = \hat{\mathbf G}_{1,n}(r_1)$
            
            \State Set $\tilde{\mathbf D}^{(1)} = \mathbf B \cdot \mathbf R^{(1)}$
            \Let{$\mathbf V^{(1)}$}{$(\mathbf R^{(1)})^\top \cdot \mathbf V^{(0)}$}
            \State Initialize $k = 1$
            \While{$k \leq n-2$}
                \State Set $l_k = -\frac{\tilde{\mathbf D}^{(k)}_{k+1,k}}{\tilde{\mathbf D}^{(k)}_{k,k}}$
                \State Construct Givens-Rotation $\mathbf L^{(k)}$ as $\mathbf L^{(k)} = \hat{\mathbf G}_{k, m}(l_k)$
                
                \Let{$\mathbf D^{(k)}$}{$\mathbf L^{(k)} \cdot \tilde{\mathbf D}^{(k)}$}
                \Let{$\mathbf U^{(k)}$}{$\mathbf U^{(k-1)} \cdot (\mathbf L^{(k)})^\top$}

                \State Set $r_{k+1} = \frac{\mathbf D^{(k)}_{k,k+2}}{\mathbf D^{(k)}_{k,k+1}}$
                \State Construct Givens-Rotation $\mathbf R^{(k+1)}$ as $\mathbf R^{(k+1)} = \hat{\mathbf G}_{k+1, n}(r_{k+1})$

                \Let{$\tilde{\mathbf D}^{(k+1)}$}{$\mathbf D^{(k)} \cdot \mathbf R^{(k+1)}$}
                \Let{$\mathbf V^{(k)}$}{$(\mathbf R^{(k+1)})^\top \cdot \mathbf V^{(k-1)}$}

                \Let{k}{k+1}
            \EndWhile
            
            \State Set $l_{n-1} = -\frac{\mathbf B_{n,n-1}}{\mathbf B_{n-1,n-1}}$
            \State Construct the last Givens-Rotation $\mathbf L^{(n-1)}$ as $\mathbf L^{(n-1)} = \hat{\mathbf G}_{n-1, m}(l_{n-1})$
            
            \Let{$\mathbf D^{(n-1)}$}{$\mathbf L^{(n-1)} \cdot \tilde{\mathbf D}^{(n-1)}$}
            \Let{$\mathbf U^{(n-1)}$}{$\mathbf U^{(n-1)} \cdot (\mathbf L^{(n-1)})^\top$}
            
            \Let{$\mathbf B$}{$\mathbf D^{(n-1)}$}
            
            \If{a diagonal element of $\mathbf B$ is zero}
                \State Use rotation step (cf. \cite{Golub1971} section 1.4 for details)
            \EndIf
            \If{a super diagonal element of $\mathbf B$ is zero}
                \State Proceed independently with two separate sub-matrices 
            \EndIf
        \EndWhile

        \Ensure{$\mathbf \Sigma = \mathbf B$ (diagonal), $\mathbf U = \mathbf Q$ (unitary) and $\mathbf V = \mathbf P^T$ (unitary) with $\mathbf A = \mathbf U \cdot \mathbf \Sigma \cdot \mathbf V^T$}
    \end{algorithmic}
\end{algorithm}

\begin{algorithm}[t] 
	\caption{\textbf{:}  \textbf{Photonics optimized QR-Decomposition}} 
    \label{alg:PhotonicQRDecomposition}
	\begin{algorithmic}[1]
		\Require{$\mathbf A \in \mathbb R^{n\times n}$ with $m \geq n$}
		\Statex
        
        \State $\mathbf A^{(0)} = \mathbf A$
        \State $\mathbf Q^{(0)} = \mathbf I_{m\times m}$
        \State Initialize $k=1$
        \While{$k \le n$}
            \State Reset PIC
            \State Initialize $l = m$
            \State Initialize $d = \mathbf A^{(k-1)}_{m,k}$
            \While{$l > k$}
                \State Compute $r_{k,l} = -\frac{d}{\mathbf A^{(k-1)}_{l-1,k}}$
                \State Program PIC to encode $\hat{\mathbf G}^n_{l - 1, l}(r_{k,l})$ at $1, l - 1$, cf.\ Figure \ref{fig:encoding_bidiag}
                \State Set $d = \frac{\mathbf A^{(k-1)}_{l-1,k} - r_{k,l} d}{\sqrt{1+r_{k,l}^2}}$
                \Let {$l$}{$l-1$}
            \EndWhile
            \State Calculate $\mathbf A^{(k)}$ by passing $\mathbf A^{(k-1)}$ column-wise through PIC
            \State Calculate $(\mathbf Q^{(k)})^T$ by passing $(\mathbf Q^{(k-1)})^T$ column-wise through PIC
        \EndWhile
		\Statex
		\Ensure{$\mathbf R = \mathbf A^{(n)}$ (upper triangular) and $\mathbf Q = \mathbf Q^{(n)}$ (unitary) with $\mathbf A = \mathbf Q \cdot \mathbf R$}
	\end{algorithmic}
\end{algorithm}

\begin{algorithm}[t] 
	\caption{\textbf{:}  \textbf{Photonics optimized GRK Bidiagonalization}} 
    \label{alg:rgkSVDforChip}
	\begin{algorithmic}[1]
		\Require{$\mathbf A \in \mathbb R^{n\times n}$ with $m \geq n$}
		\Statex
        
        \State $\mathbf B^{(0)} = \mathbf A$
        \State $\mathbf P^{(0)} = \mathbf I_{m\times m}$
        \State $\mathbf Q^{(0)} = \mathbf I_{n\times n}$
        \State Initialize $k=1$
        \While{$k \le n$}
            \State Reset PIC
            \State Initialize $l = m$
            \State Initialize $d = \mathbf B^{(k-1)}_{m,k}$
            \While{$l > k$}
                \State Compute $r_{k,l} = -\frac{d}{\mathbf B^{(k-1)}_{l-1,k}}$
                \State Program PIC to encode $\hat{\mathbf G}_{l - 1, l}^m(r_{k,l})$ at $1, l - 1$, cf.\ Figure \ref{fig:encoding_bidiag}
                \State Set $d = \frac{\mathbf B^{(k-1)}_{l-1,k} - r_{k,l} d}{\sqrt{1+r_{k,l}^2}}$
                \Let {$l$}{$l-1$}
            \EndWhile
            \State Calculate $\tilde{\mathbf B}^{(k)}$ by passing $\mathbf B^{(k-1)}$ column-wise through PIC
            \State Calculate $(\mathbf P^{(k)})^T$ by passing $(\mathbf P^{(k-1)})^\top$ column-wise through PIC
            \State Reset PIC
            \State Initialize $l=n$
            \State Set $d = \mathbf B^{(k-1)}_{k,n}$
            \While{$l > k+1$}
                \State Compute $s_{k,l} = \frac{d}{\mathbf B^{(k-1)}_{k,l-1}}$
                \State Program PIC to encode $\hat{\mathbf G}_{l - 1, l}^n(s_{k,l})$ at $1, l - 1$
                \State Set $d = \frac{\mathbf B^{(k-1)}_{k, l-1} - s_{k,l} d}{\sqrt{1+s_{k,l}^2}}$
                \Let {$l$}{$l-1$}
            \EndWhile
            \State Calculate $(\mathbf B^{(k)})^\top$ by passing $(\tilde{\mathbf B}^{(k)})^\top$ column-wise through PIC
            \State Calculate $\mathbf Q^{(k)}$ by passing $\mathbf Q^{(k-1)}$ column-wise through PIC
            \Let {$k$}{$k+1$}
        \EndWhile
		\Statex
		\Ensure{$\mathbf B = \mathbf B^{(n)}$ (bidiagonal), $\mathbf Q = \mathbf Q^{(n)}$ (unitary) and $\mathbf P = \mathbf P^{(n)}$ (unitary) with $\mathbf A = \mathbf P \cdot \mathbf B \cdot \mathbf Q$}
	\end{algorithmic}
\end{algorithm}

\begin{algorithm}[t]
    \caption{\textbf{:}  \textbf{Photonics optimized GRK-SVD (GRK-SVD H)}}
    \label{alg:PhoRGKSVD}
    \begin{algorithmic}[1]
        \Require{$\mathbf A \in \mathbb R^{m\times n}$ with $m \geq n$. (Note that this Algorithm needs the PIC in a flipped configuration, such that the first two input channels interact first.)}
		\Statex
        \State Decompose $\mathbf A = \mathbf Q \cdot \mathbf B \cdot \mathbf P$ via Algorithm \ref{alg:rgkSVDforChip}
        \State Set $\delta = \epsilon_0 \lVert \mathbf B\rVert_\infty$ where $\epsilon_0$ is the machine precision
        \While{$\lVert \mathbf B - $diag($\mathbf B$)$ \rVert_\infty > \delta$}
            \State Set the Wilkinson Shift $s$ as the dominant eigenvalue of the bottom right 2x2 minor of $\mathbf B^T \mathbf B$:
            \begin{align*}
                s = \max(\sigma(\mathbf{M})), \quad \text{for} \quad \mathbf M =
                \begin{bmatrix}
                    \mathbf B_{n-2, n-1}^2+\mathbf B_{n-1, n-1}^2   & \mathbf B_{n-1, n-1}\mathbf B_{n-1, n}   \\
                    \mathbf B_{n-1, n-1}\mathbf B_{n-1, n} & \mathbf B_{n-1, n}^2+\mathbf B_{n, n}^2
                \end{bmatrix},
            \end{align*}
            \State Set the desired ratio $r_1$ of the first two entries in the first column of $\mathbf B^T \mathbf B - s \mathbf I_{n \times n}$ to
            \begin{align*}
                r_1 = \frac{\mathbf B_{1,2}\mathbf B_{1,1}}{\mathbf B_{1,1}^2-s}
            \end{align*}
            \State Construct the first Givens-Rotation $\mathbf R^{(1)}$ as $\mathbf R^{(1)} = \hat{\mathbf G}_{1,n}(r_1)$
            
            \State Set $\tilde{\mathbf D}^{(1)} = \mathbf B \cdot \mathbf R^{(1)}$
            \State Initialize $k = 1$
            \While{$k \leq n-2$}
                \State Set $l_k = -\frac{\tilde{\mathbf D}^{(k)}_{k+1,k}}{\tilde{\mathbf D}^{(k)}_{k,k}}$
                \State Construct Givens-Rotation $\mathbf L^{(k)}$ as $\mathbf L^{(k)} = \hat{\mathbf G}_{k, m}(l_k)$
                
                \Let{$\mathbf D^{(k)}$}{$\mathbf L^{(k)} \cdot \tilde{\mathbf D}^{(k)}$}

                \State Set $r_{k+1} = \frac{\mathbf D^{(k)}_{k,k+2}}{\mathbf D^{(k)}_{k,k+1}}$
                \State Construct Givens-Rotation $\mathbf R^{(k+1)}$ as $\mathbf R^{(k+1)} = \hat{\mathbf G}_{k+1, n}(r_{k+1})$

                \Let{$\tilde{\mathbf D}^{(k+1)}$}{$\mathbf D^{(k)} \cdot \mathbf R^{(k+1)}$}

                \Let{k}{k+1}
            \EndWhile
            
            \State Set $l_{n-1} = -\frac{\mathbf B_{n,n-1}}{\mathbf B_{n-1,n-1}}$
            \State Construct the last Givens-Rotation $\mathbf L^{(n-1)}$ as $\mathbf L^{(n-1)} = \hat{\mathbf G}_{n-1, m}(l_{n-1})$

            \Let{$\mathbf D^{(n-1)}$}{$\mathbf L^{(n-1)} \cdot \tilde{\mathbf D}^{(n-1)}$}

            \State Encode all inverses of the matrices $\mathbf{R}^{(k)}$ on the PIC, where $(\mathbf{R}^{(k)})^T$ is put at $1, k$ 
            \State Calculate $\tilde{\mathbf{P}}$ by passing $\mathbf{P}$ column-wise through PIC which is now configured as in Figure \ref{fig:encoding_qriteration}
            
            \State Encode all inverses of the matrices $\mathbf{L}^{(k)}$ on the PIC, where $(\mathbf{L}^{(k)})^T$ is put at $1, k$ 
            \State Calculate $\tilde{\mathbf{Q}}^T$ by passing $\mathbf{Q}^T$ column-wise through PIC, cf.\ Figure \ref{fig:encoding_qriteration}
            
            \Let{$\mathbf{P}$, $\mathbf{Q}$}{$\tilde{\mathbf{P}}$, $\tilde{\mathbf{Q}}$}

            \Let{$\mathbf B$}{$\mathbf D^{(n-1)}$}
            
            \If{a diagonal element of $\mathbf B$ is zero}
                \State Use rotation step (cf. \cite{Golub1971} section 1.4 for details)
            \EndIf
            \If{a super diagonal element of $\mathbf B$ is zero}
                \State Proceed independently with two separate sub-matrices 
            \EndIf
        \EndWhile

        \Ensure{$\mathbf \Sigma = \mathbf B$ (diagonal), $\mathbf U = \mathbf Q$ (unitary) and $\mathbf V = \mathbf P^T$ (unitary) with $\mathbf A = \mathbf U \cdot \mathbf \Sigma \cdot \mathbf V^T$}
    \end{algorithmic}
    
\end{algorithm}


\bibliography{references}
\bibliographystyle{plain}

\end{document}